\documentclass[12pt, oneside]{article}

\usepackage[utf8]{inputenc}
\usepackage[english]{babel}
\usepackage{amsthm}
\usepackage{amsmath}
\usepackage{amssymb}
\usepackage{amsfonts}

\usepackage[usenames]{color}

\usepackage[final]{graphicx}

\usepackage{titling}
\usepackage{abstract}
\usepackage{hyperref}

\usepackage{cite}

\theoremstyle{plain}
\newtheorem{theorem}{Theorem}[section]

\newtheorem{lemma}[theorem]{Lemma}
\newtheorem{corollary}[theorem]{Corollary}

\newtheorem*{theorem*}{Theorem}
\newtheorem*{lemma*}{Lemma}

\theoremstyle{definition}
\newtheorem{definition}[theorem]{Definition}
\newtheorem{remark}[theorem]{Remark}

\newcommand{\Lip}{ {\rm Lip} }
\newcommand{\Adj}{ \operatorname {Adj} }
\newcommand{\tr}{ \operatorname {tr}  }
\newcommand{\loc}{ {\rm loc} }
\newcommand{\R}{ \mathbb{R} }
\newcommand{\M}{ \mathbb{M} }
\newcommand{\A}{ \mathcal{A} }
\newcommand{\Hom}{ \mathcal{H} }
\newcommand{\K}{ \mathcal{K} }
\newcommand{\N}{ \mathcal{N} }

\def\?{
?\vadjust{\vbox to 0pt{\vss\hbox{\kern\hsize\kern1em\large\bf ?!}}}}

\title{
Injectivity almost everywhere 
and
mappings with finite distortion
in nonlinear elasticity 
\thanks{This work was partially supported by a
	Grant of the Russian Foundation of Basic Research 
	(Project~17-01-00801-a) 
	and
	the Russian Science Foundation
	(Agreement  No.~16-41-02004).}
}
\author{A.~O.~Molchanova \\
	\normalsize\textit{Sobolev Institute of Mathematics, People's Friendship University,} \\
	\normalsize \textit{\href{mailto:a.molchanova@math.nsc.ru}{a.molchanova@math.nsc.ru}}
	\and
	S.~K.~Vodop$'$yanov \\
	\normalsize\textit{Sobolev Institute of Mathematics, People's Friendship University,} \\
	\normalsize \textit{\href{mailto:vodopis@math.nsc.ru}{vodopis@math.nsc.ru}}
}
\date{\empty}
\renewcommand{\maketitlehookd}{%
\begin{abstract} 
	\noindent 
	We show that a sufficient condition for the  
	weak limit of a~sequence of 
	$W^1_q$-homeomorphisms with finite distortion to be 
	almost everywhere injective
	for 
	$q \geq n-1$,
	can be stated 
	by means of composition operators.
	Applying this result, we study nonlinear elasticity problems 
	with respect to these new classes of mappings.
	Furthermore, we impose loose growth conditions on the stored-energy function for the class of 
	$W^1_n$-homeomorphisms with finite distortion and 
	integrable inner as well as outer distortion coefficients.

	{\it Keywords:} finite distortion,
	functional minimization problem,
	injectivity almost-everywhere,
	nonlinear elasticity,
	polyconvexity.

	{\it Mathematics Subject Classification} (2010): 
	Primary: 30C65; 
	Secondary:   
		46E35  	
		74B20. 
\end{abstract}
}

\begin{document}

\maketitle

\section*{Introduction}\label{sec:intro}
Some problems in nonlinear elasticity
(including, for instance,
those involving 
hyperelastic materials)
reduce to 
that of
minimizing the total energy functional. 
In this situation,
and
in contrast to the case of linear elasticity,
the integrand is almost always nonconvex,
while the functional is nonquadratic.
This renders the standard variational methods inapplicable.
Nevertheless,
for a~sufficiently large class of applied nonlinear problems,
we may replace convexity with certain weaker conditions,
i.e.\ 
polyconvexity \cite{Ball1977}.

Denote by
$\M^{m\times n}$ 
the set of
$m\times n$
matrices.
Recall that 
a function 
$W\colon \Omega \times \M^{3\times 3} \to \R$,
$\Omega \subset \R^3$, 
is called \textit{polyconvex} if
there exists a~convex function
${G(x,\cdot)\colon\M^{3\times 3}\times\M^{3 \times 3} \times\R_{+} \to \R}$
such that
\begin{equation*}
	G(x, F, \Adj F, \det F)=W(x,F)
	\text{ for all } F\in \M^{3\times 3} \text{ with } \det F > 0,
\end{equation*}
almost everywhere (henceforth abbreviated as a.e.) in~%
$\Omega$.

Let 
$\Omega$
be a bounded domain in 
$\R^3$
which boundary
$\partial \Omega$
satisfies the Lipschitz condition.
Ball's method~\cite{Ball1977}
is to consider a~sequence
$\{\varphi_k\}_{k\in\mathbb{N}}$
minimizing the total energy functional
\begin{equation}\label{def:energy_func}
	I(\varphi)=\int\limits_{\Omega}W(x,D\varphi)\,dx.
\end{equation}
over the set of admissible deformations
\begin{equation}\label{def:AB}
	\A_B=\{\varphi\in W^1_1(\Omega), \: I(\varphi) < \infty, \:
	J(x,\varphi) > 0 
	\text{ a.e.~in } \Omega, \:
	\varphi|_{\partial \Omega}=\overline{\varphi}|_{\partial \Omega}\},
\end{equation}
where
$\overline{\varphi}$
are Dirichlet boundary conditions
and 
$J(x,\varphi)$ 
stands for the Jacobian of
$\varphi$,
$J(x,\varphi) = \det D\varphi (x)$.
Furthermore, it is
assumed that
the coercivity inequality
\begin{equation}\label{neq:coer_b}
	W(x,F)\geq \alpha (|F|^p+|\Adj F|^{q}+ (\det F)^r) + g(x)
\end{equation}
holds for almost all
$x\in \Omega$
and all
$F\in \M^{3\times 3}$,
$\det F > 0$,
where
$p\geq 2$,
$q\geq \frac{p}{p-1}$,
$r>1$
and 
$g\in L_1(\Omega)$,
$\Adj F$
denotes the adjoint matrix, i.e.\ a~transposed matrix of 
$(2 \times 2)$-subdeterminants
of 
$F$.  
Moreover,
the stored-energy function~%
$W$
is polyconvex.
By coercivity,
it follows that
the sequence
$(\varphi_k,\Adj D\varphi_k, \det D\varphi_k)$
is bounded in the reflexive Banach space
$W^1_p(\Omega)\times L_{q}(\Omega)\times L_r(\Omega)$.
Relying on the relation between 
$p$
and 
$q$, 
one can conclude that 
there exists a~subsequence converging weakly to an~element
$(\varphi_0,\Adj D\varphi_0, \det D\varphi_0)$.
For the limit
$\varphi_0$
to belong to the class
$\A_B$
of admissible deformations,
we need to impose the additional condition: 
\begin{equation} \label{cond:barrier}
	W(x,F)\rightarrow\infty
	\text{ as }
	\det F \rightarrow 0_+
\end{equation}
(see~\cite{BallCurOl1981} for more details). 
This condition is quite reasonable 
since it fits in with the principle that
{\it ``extreme stress must accompany extreme strains''.}
Another important property of this approach
is the sequentially weakly lower semicontinuity
of the total energy functional,
\begin{equation*}
	I(\varphi)\leq \varliminf\limits_{k\rightarrow\infty} I(\varphi_k),
\end{equation*}
which holds
because the stored-energy function is polyconvex.
It is also worth noting that
Ball's approach admits the {\it nonuniqueness of solutions}
observed experimentally
(see~\cite{Ball1977} for more details).

One of the most important requirements of continuum mechanics is 
that interpenetration of matter does not occur,
from which it follows that any deformation has to be injective.
Global injectivity of deformations has been established by J.~Ball
\cite{Ball1981} 
within the existence theory based on minimization of the energy 
\cite{Ball1977}.
More precisely,
if 
$\varphi \colon \overline\Omega \to \mathbb{R}^n$,
$\Omega \subset \R^n$,
is a mapping in
$W^{1}_{p}(\Omega)$,
$p > n$, 
coinciding on the boundary 
$\partial \Omega$
with a homeomorphism
$\overline\varphi$
and
$J(x, \varphi) > 0$ 
a.e.\ in 
$\Omega$,
$\overline\varphi(\Omega)$
is Lipschitz,
and if for some 
$\sigma > n$
\begin{equation}\label{eq:ball_inj}
	\int\limits_\Omega |(D \varphi(x))^{-1}|^\sigma J(x, \varphi) \, dx =
	\int\limits_\Omega \frac{|\Adj D \varphi(x)|^\sigma} {J(x, \varphi)^{\sigma-1}} \, dx
	< \infty,
\end{equation}
then 
$\varphi$ 
is a homeomorphism of 
$\Omega$
on 
$\varphi(\Omega)$
and
$\varphi^{-1}\in W^1_\sigma(\overline\varphi(\Omega))$.

To apply this result to nonlinear elasticity
it is required that some
additional conditions on 
the stored-energy function
be imposed in order to
obtain 
invertibility of deformations.
Thus,
in~\cite{Ball1981}
(see also~\cite[Exercise~7.13]{Ciar1988}),
it is considered a~domain
$\Omega \subset \R^3$
with a Lipschitz boundary
$\partial \Omega$
and a~polyconvex stored-energy function
$W$.
Suppose that
there exist constants
$\alpha>0$,
$p>3$,
$q > 3$,
$r>1$,
and
$m> \frac{2q}{q-3}$,
as well as a~function
$g\in L_1(\Omega)$
such that
\begin{equation}\label{neq:qoerB}
	W(x,F)\geq \alpha (|F|^p + |\Adj F|^{q}+ (\det F)^r + (\det F)^{-m}) + g(x)
\end{equation}
for almost all
$x\in \Omega$
and all
$F\in \M^{3\times 3}$,
$\det F > 0$.
Take a~homeomorphism
$\overline{\varphi}\colon \overline{\Omega} \to \overline{\Omega'}$
in
$W^1_p(\Omega)$
with
$J(x,\overline\varphi) > 0$
a.e.\ in 
$\Omega$.
Then there exists a~mapping
$\varphi\colon \Omega \to \Omega'$
minimizing the total energy functional~\eqref{def:energy_func} 
over the set of admissible deformations~\eqref{def:AB},
which is a~homeomorphism due to~\eqref{eq:ball_inj} with 
$\varphi^{-1}\in W^1_\sigma(\overline\Omega)$,
$\sigma = \frac{q(1+m)}{q+m} > 3$. 

In this article we obtain the injectivity property
(Theorem~\ref{thm:ae_injectivity})
based on the boundedness of the composition operator
$\varphi^*\colon L^1_p(\Omega')\to L^1_q(\Omega)$.
Boundedness of these operators is intimately related to a condition of finite distortion.
Recall that 
a
$W^1_{1, \loc}$-mapping
$f\colon\Omega\to \R^n$
with nonnegative Jacobian,
$J(x,f) \geq 0$
a.e.,
is called a~{\it mapping with finite distortion}
if
$|Df(x)|^n \leq K(x) J(x,f)$ 
for almost all 
$x\in \Omega$,
where
$1 \leq K(x) < \infty$
a.e.\ in~%
$\Omega$.
A~function
\begin{equation*}\label{def:outer_distortion}
	K_O (x,\varphi) = \frac{|D\varphi(x)|^n}{J(x,\varphi)}
\end{equation*}
is called the \textit{outer distortion coefficient}%
\footnote{
	It is assumed that
	$K_O (x,\varphi) = 1$ 
	if
	$J(x,\varphi) = 0$.
}.
It is worth noting that
mappings with finite distortion arise in  nonlinear elasticity from geometric considerations:
it would be desirable that the deformation is continuous, maps sets of measure zero to sets of measure zero,
is a one-to-one mapping and that the inverse map has ``good'' properties.
Hence, many research groups all over the word have worked on this issue
(see 
\cite{AstIwaMarOnn2005,BenKru2016,BenKam2015,HeiKos1993,HenKos2005,HenKos2006,HenKos2014,HenMal2002,IwaOnn2009,IwaOnn2010,IwaOnn2011,IwaOnn2012,KosMal2003,ManVill1998,MarMal1995,Onn2006,VodMol2015} 
and a lot more).
It is known that in the planar case 
($\Omega$,
$\Omega'\subset \R^2$) 
a homeomorphism 
$\varphi \in W^{1}_{1,\loc}(\Omega)$ 
has an inverse homeomorphism
$\varphi^{-1} \in W^{1}_{1,\loc}(\Omega')$
if and only if 
$\varphi$
is a mapping with finite distortion
\cite{HenKos2006,HenKos2014}.
In the spatial case
$W^1_{n,\loc}$-regularity 
of the inverse mapping
was shown
for 
$W^1_{q,\loc}$-homeomorphism, 
$q > n-1$,
with the integrable \textit{inner distortion}%
\footnote{
	Here
	$K_I (x,\varphi) = 1$ 
	if
	$|\Adj D\varphi (x)| = 0$,
	and
	$K_I (x,\varphi) = \infty$ 
	if
	$|\Adj D\varphi (x)| \neq 0$
	and
	$J(x,\varphi) = 0$.
	}
\begin{equation*}
	K_I(x,\varphi) = \frac{|\Adj D \varphi (x)|^n}{J(x,\varphi)^{n-1}}.
\end{equation*}
Moreover, the relaxation of \eqref{eq:ball_inj}
on the case 
$\sigma = n$,
\begin{equation*}
	\int\limits_{\Omega'} |D \varphi^{-1}(y)|^n  \, dy
	= \int\limits_\Omega \frac{|\Adj D \varphi (x)|^n}{J(x,\varphi)^{n-1}} \, dx
	= \int\limits_\Omega K_I(x,\varphi) \, dx,
\end{equation*}
holds \cite{Onn2006,Vod2012}.

	In \cite{IwaOnn2009,IwaOnn2010,IwaOnn2011,IwaOnn2012} the authors 
	study 
	$W^1_n$-homeomorphisms
	$\varphi\colon \Omega \to \Omega'$
	between two bounded domains in 
	$\mathbb{R}^n$ 
	with finite 
	energy
	and consider the behavior of such mappings.
	In general, the weak
	$W^1_n$-limit of a sequence of homeomorphisms
	may lose injectivity.
	However, if there is a requirement on 
	totally boundedness of norms of the inner distortion
	$\|K_I(\cdot,\varphi) \mid L_1(\Omega) \|$
	and some additional requirements,
	then the limit map is a homeomorphism.
	The main idea behind the proof of existence and global invertibility 
	is to investigate admissible deformations 
	$\varphi_k$
	in parallel with
	its inverse
	$\varphi_k^{-1}$
	along to a minimizing sequence
	$\{\varphi_k\}$.
	This is possible\footnote{See \cite{Vod2012} for  another proof of this property   under weaker assumption}
	 due to integrability of the inner distortion as this ensures the existence and regularity 
	of an inverse map
	belonging to 
	$W^1_n$.
	Note that the authors of these papers include requirements of integrability 
	of the inner distortion coefficient in the coercive inequality.
	The authors of the current paper prefer to include this condition 
	to the class of admissible deformation, so as to obtain more ``\textit{fine graduation}'' 
	of deformations.

We also emphasize that the aforementioned regularity properties of an inverse homeomorphism 
(including the case 
$q = n-1$)
can be obtained using
a technique of the theory of bounded operators of Sobolev spaces.
Putting
$p = \frac{\sigma(n-1)}{\sigma-1}$,
$p' = \sigma$,
$q = n-1$,
$q' = \infty$,
$\varrho = \sigma$
in Theorem~\ref{thm:Vod3-4} \cite[Theorem 3]{Vod2012}
we derive the aforementioned result from \cite{Ball1981}.
By taking
$p = p' = \varrho = n$,
$q = n-1$,
$q' = \infty$
in the same theorem
one can obtain the regularity of an inverse mapping from \cite{Onn2006}.

Whereas we have dealing with 
$W^1_n$-mappings 
with finite distortion in this article,
we  reduce 
coercivity conditions on the stored-energy function to 
\begin{equation} \label{neq:coer1}
	W(x,F)\geq \alpha |F|^n + g(x).
\end{equation}
For given constants 
$p$,
$q\geq 1$
and
$M > 0$,
and the total energy 
$I$,
identified by~\eqref{def:energy_func},
we define
the class of admissible deformations
\begin{multline*}
	\Hom(p,q,M)=\{
	\varphi \colon \Omega \to \Omega' \text{ is a homeomorphism with finite distortion, }
	\\
	\varphi \in W^1_1(\Omega),\:
	I(\varphi) < \infty, \:
	\: J(x,\varphi) \geq 0 
	\text{ a.e.~in } \Omega, \: 
	\\
	K_O(\cdot, \varphi) \in L_{p} (\Omega), \:
	\| K_I(\cdot,\varphi) \mid L_{q} (\Omega)\| \leq M \},
\end{multline*}
where
$K_O(x,\varphi)$
and
$K_I(x,\varphi)$
are the outer and the inner distortion coefficients.
We prove an existence theorem in  the following formulation
(see precise requirements
in Section~\ref{subsec:exist}).

\begin{theorem*}[Theorem~\ref{thm:main} below]
	Let
	$\Omega$, 
	$\Omega'\subset \R^n$
	be bounded domains
	with Lipschitz boundaries.
	Given a~polyconvex function
	$W(x,F)$,
	satisfying the coercivity inequality~\eqref{neq:coer1}, 
	and a~nonempty set~%
	$\Hom(n-1,s,M)$
	with
	$M > 0$,
	$s > 1$,
	then there exists at least one homeomorphic mapping
	\begin{equation*}
		\varphi_0\in\Hom(n-1,s,M)\quad\text{such that}\quad
		I(\varphi_0)=\inf\limits\{I(\varphi),\varphi\in\Hom(n-1,s,M)\}.
	\end{equation*}
\end{theorem*}

The existence theorem is also obtained 
for classes  
of mappings with prescribed boundary values
and the same homotopy class as a given one,
and covers the case 
$s=1$
in some cases (Section~\ref{subsec:exist}).
Note that 
	the class of admissible deformations from the paper \cite{IwaOnn2009} 
	is related to  
	considered in the present paper classes
	(see Remark~\ref{rem:free_outer_dist}).
	For the same reason, the elasticity result of \cite{Ball1981} can be derived from the result of the present paper.
	Indeed, the integrability of the distortion coefficient follows from the H\"{o}lder inequality and \eqref{neq:qoerB}
	by 
	$s = \frac{\sigma r}{rn + \sigma - n}$
	where 
	$\sigma = \frac{q(1+m)}{q+m}$
	(see Section~\ref{sec:examples}).

Some important properties of mappings of these classes can be found in \cite{Vod}.
Note also that
the property of mapping to be sense preserving
in the topological way
follows from the property that
the required deformation is a~mapping with bounded
$(n,q)$-distortion
if
$q> n-1$
\cite[Remark 1]{BayVod2015}.

Additionally, there is a different approach to injectivity which was proposed by
P.~Ciarlet and I.~Ne\v{c}as in~\cite{CiarNec1987}.
This approach rests upon the additional
{\it injectivity condition}
\begin{equation}\label{cond:injectivity}
	\int\limits_{\Omega} J(x,\varphi)\, dx \leq |\varphi(\Omega)|
\end{equation}
on the admissible deformations
if 
$\Omega \subset \R^n$
is a
bounded open set
with
$C^1$-smooth
boundary,
$\varphi \in W^1_p (\Omega)$,
$p > n$,
and
$J(x,\varphi) > 0$
a.e.\
in 
$\Omega$.
Under these assumptions, the minimization problem of the energy functional can be constrained to a.e.\ injective deformations. 
In the three-dimensional case the relation~\eqref{cond:injectivity}
under the weaker hypothesis
$p > n-1$
was studied in
\cite{Tan1988}.
In this case, 
$\varphi$
may no longer be
continuous
and the
inverse mapping
$\varphi^{-1}$
has only regularity 
$BV_\loc (\varphi (\Omega), \mathbb{R}^n)$.
Local invertibility properties of the mapping
$\varphi \in W^1_p (\Omega)$,
$p \geq n$,
under the condition
$J(x,\varphi) > 0$
a.e.,
can be found in \cite{FonGan1995}.
The case 
$p > n - 1$
is considered in the recent paper \cite{BarHenMor2017},
the approach of which uses the topological degree as an essential tool
and based on some ideas of \cite{MulSpe1995}. 
Some other studies of local and global invertibility  
in the context of elasticity
can be found in~%
\cite{BauPhi1994,CiarNec1985,ConLel2003,HenMor2010,HenMor2015,IwaOnn2012,MulSpe1995,MulSpeTan1996,Sve1988,SwaZie2002,SwaZie2004}. 
Also, see \cite{Ball2002,Ball2010} 
for a general review of research in the elasticity theory.

We will now give an outline of the paper. The first section contains general auxiliary facts
and some facts about mappings with finite distortion.
The second section is devoted to the injectivity almost everywhere property
(Theorem~\ref{thm:ae_injectivity}).
This property follows from jointly boundedness of pullback operators 
defined by a sequence of homeomorphisms 
$\varphi_k$
and  the
uniform convergence of 
inverse homeomorphisms
$\psi_k$
(Lemma~\ref{lem:prop_inv}).
Moreover, as a consequence, we obtain the strict inequality
$J(x, \varphi_0)>0$ 
a.e.\
(Lemma \ref{lem:J>0}).
The third section is dedicated to the existence theorem.
In the forth section we give two examples to illustrate advantages of our method.
Appendix contains some discussion about geometry of domains that does not direct bear on the subject of this paper but is of independent sense.

Some ideas of this article 
were announced in the note~\cite{VodMol2015}.


\section{Mappings with finite distortion}

Mappings with finite distortion is a natural generalization of 
mappings with bounded distortion.
The reader not familiar with mappings with bounded distortion
may look at~\cite{Resh1982,Rick1993}.
To 	take a close look at the theory of mapping with bounded distortion,
the reader can study 
monographs~\cite{HenKos2014,IwaMar2001}.

In this section we present some important concepts and statements 
necessary to proceed.
On a~bounded domain
$\Omega\subset\R^n$,
i.e.\ a nonempty, connected, and open set,
we define in the~standard way
(see~\cite{Maz2011} for instance) 
the space
$C_0^{\infty}(\Omega)$
of smooth 
functions with compact support,
the~Lebesgue spaces
$L_p(\Omega)$
and
$L_{p, \loc}(\Omega)$
of integrable functions,
and Sobolev spaces
$W^1_p(\Omega)$
and
$W^1_{p, \loc}(\Omega)$,
$1 \leq p \leq \infty$.
A mapping 
$f \in L_{1,\loc}(\Omega)$
belongs to 
{\it homogeneous Sobolev class}
$L^1_p(\Omega)$,
$p \geq 1$,
if it has the weak derivatives of the first order
and its differential
$D f(x)$
belongs to
$L_p (\Omega)$.

\begin{definition}\label{def:Lip_boundary}
	We say that a~bounded domain
	$\Omega\subset\R^n$
	has a \textit{Lipschitz boundary}
	if for each
	$x \in \partial \Omega$
	there exists a neighborhood 
	$U$
	such that
	the set 
	$\Omega \cap U$
	is represented by the inequality
	$\xi_n < f(\xi_1, \dots \xi_{n-1})$
	in some Cartesian coordinate system 
	$\xi$
	with Lipschitz continuous function
	$f\colon \R^{n-1} \to \R$.

	Domains with Lipschitz boundary are sometimes called domains having 
	\textit{the strong Lipschitz property},
	whereas 
	\textit{Lipschitz domains}
	are defined through quasi-isometric mappings.
	Detailed discussion see in Appendix~\ref{sec:geometry}.
\end{definition}

Recall that for 
topological spaces
$X$
and
$Y$,
a continuous mapping 
$f\colon X \to Y$ 
is
{\it discrete} 
if 
$f^{-1}(y)$ 
is a discrete set for all 
$y \in Y$
and 
$f$ 
is 
{\it open} 
if it takes open sets onto open sets.

\begin{definition}[\hspace{-.3pt}\cite{IwaSve1993,VodGold1976}]\label{def:FD}
	Given an open set 
	$\Omega \subset \R^n$
	and
	a~mapping
	$f\colon\Omega\to \R^n$
	with
	$f \in W^1_{1, \loc}(\Omega)$
	is called a~{\it mapping with finite distortion},
	whenever
	\begin{equation*}
		|Df(x)|^n \leq K(x) |J(x,f)| \quad
		\text{for almost all}\: x\in \Omega,
	\end{equation*}
	where
	$1 \leq K(x) < \infty$
	a.e.\ in 
	$\Omega$%
	\footnote{
		Some authors include condition 
		$J(x,f) \geq 0$
		in Definition~\ref{def:FD}.
		We do not use the condition for the Jacobian to be non-negative as it is unnecessary in the context of the theory of composition operators, see details in \cite{Vod2012}.}.
\end{definition}

In other words,
the {\it finite distortion} condition 
amounts to the~vanishing of the partial derivatives of 
$f \in W^1_{1, \loc}(\Omega)$
almost everywhere on the zero set of the Jacobian
$Z = \{x \in \Omega: J(x,f) = 0\}$. 
Similarly, the {\it finite codistortion} condition 
means that 
$\Adj Df (x) = 0$ 
a.e.\ on the the set 
$Z$.
If 
$K \in L_\infty(\Omega)$, 
a mapping
$f$ 
is called a~\textit{mapping with bounded distortion} (or a~\textit{quasiregular mapping}).

For a mapping with finite distortion with 
$J(x,f) \geq 0$
a.e.\
the functions
\begin{equation}\label{def:outer_inner_distortion}
	K_O(x,f)=\frac{|Df(x)|^n}{J(x,f)} 
	\quad \text{and} \quad
	K_I(x,f)=\frac{|\Adj Df(x)|^n}{J(x,f)^{n-1}}
\end{equation}
when 
$0 < J(x,f) < \infty$
and 
$K_O(x,f) = K_I(x,f) = 1$ 
otherwise
are called the {\it outer} and the {\it inner distortion} coefficients of~%
$f$
at the point~%
$x$.
It is easy to see that
\begin{equation*}
	K_I^{\frac{1}{n-1}}(x,f) \leq K_O(x,f) \leq K_I^{n-1}(x,f)
	\quad\text{for a.e.} \:
	x\in \Omega. 
\end{equation*}

In 1967 Yu.~Reshetnyak proved strong topological properties of mappings with bounded distortion: continuity, openness, and discreteness
\cite{Resh1967-2}. 
Theorem~2.3 of~\cite{VodGold1976} shows that
$W^1_{n,\loc}$-mapping 
with finite distortion and nonnegative Jacobian, 
$J(x,f) \geq 0$ a.e.,
is continuous. 

In recent years, a lot of research has been done in order to find the sharp assumptions 
for these topological properties in the class of mappings with finite distortion, 
for example, \cite{HeiKos1993,HenKos2005,HenMal2002,IwaSve1993,ManVill1998}.

\begin{theorem}[{\hspace{-.3pt}\cite{Raj2010}}]\label{thm:Raj}
	Let 
	$f\colon \Omega \to \mathbb{R}^n$, 
	$n\geq 2$,
	be a non-constant mapping with finite distortion satisfying
	$J(x,f) \geq 0$ a.e.,
	$f\in W^1_{n,\loc}(\Omega)$,
	$K_O(\cdot,f)\in L_{n-1, \loc}(\Omega)$
	and 
	$K_I(\cdot,f)\in L_{s, \loc}(\Omega)$
	for some
	$s > 1$.
	Then
	$f$
	is discrete
	and open. 
\end{theorem}

On the other hand, mappings with finite distortion are closely related to 
boundedness of composition operators of Sobolev spaces.
Recall that 
a~measurable mapping 
$\varphi\colon\Omega\to \Omega'$ 
induces {\it a~bounded operator}
$\varphi^*\colon L^1_p(\Omega^{\prime})\to L^1_q(\Omega)$ 
{\it by the composition rule}, 
$1\leq q\leq p<\infty$, 
if 
the operator
$\varphi^*\colon L^1_p(\Omega^{\prime}) \cap {\Lip}_{\loc} (\Omega^{\prime})\to L^1_q(\Omega)$ 
with
$\varphi^*(f)=f\circ \varphi$,
$f\in  L^1_p(\Omega^{\prime}) \cap {\Lip}_{\loc}(\Omega^{\prime})$,
is bounded.

\begin{lemma}[\hspace{-.3pt}\cite{VodUhl2002}]\label{lem:FD}
	If a measurable mapping 
	$\varphi$
	induces a bounded composition operator
	\begin{equation*}
		\varphi^*\colon L^1_p(\Omega') \to L^1_q (\Omega), 
		\quad 1\leq q \leq p \leq \infty,
	\end{equation*}
	then 
	$\varphi$
	has finite distortion.
\end{lemma}

Now we consider a generalization of inner and outer distortion functions,
which is more conducive to dealing with composition and pullback operators.
Following~\cite{Vod2012},
for a~mapping
$f\colon\Omega\to\Omega'$
of class
$W^1_{1,\loc}(\Omega)$
define the (\textit{outer}) \textit{distortion operator function}
\begin{equation*}\label{def:dist_op_func}
	K_{f,p} (x) =
	\begin{cases}
		\dfrac{|Df(x)|}{|J(x,f)|^{{1}/{p}}} &
		\text{ for } 
		x\in \Omega\setminus Z,\\
		0 &
		\text{ otherwise},
	\end{cases}
\end{equation*}
and the (\textit{inner}) \textit{distortion operator function}
\begin{equation*}\label{def:codist_op_func}
	\mathcal{K}_{f,p} (x) =
	\begin{cases}
		\dfrac{|\Adj Df(x)|}{|J(x,f)|^{(n-1)/{p}}} &
		\text{ for } 
		x\in \Omega\setminus Z,\\
		0 &
		\text{ otherwise},
	\end{cases}
\end{equation*}
where~%
$Z$
is a zero set of the Jacobian
$J(x,f)$.

\begin{remark}\label{rem:connection_distortion}
	Note that 
	$K_O(x,f) = K_{f,n}^n(x)$
	and 
	$K_I(x,f) = \K_{f,n}^n(x)$
	if 
	$x \in \Omega \setminus Z$.
	Hence 
	$K_O(\cdot,f)\in L_{n-1} (\Omega)$
	results in 
	$K_{f,n}(\cdot) \in L_{n(n-1)} (\Omega)$,
	and
	$\|K_I(\cdot,f) \mid L_s(\Omega)\| \leq M $
	implies
	$\|\K_{f,n}(\cdot) \mid L_\varrho(\Omega)\| \leq M^{1/n} $
	for 
	$\varrho = ns$.
\end{remark}

The following theorem shows the regularity properties which ensure 
that the direct and the inverse homeomorphisms belong to 
corresponding Sobolev classes.

\begin{theorem}[{\hspace{-.3pt}\cite[Theorem 3]{Vod2012}}] 
\label{thm:Vod3-4}
	Let 
	$\varphi\colon\Omega\to\Omega'$
	be a~homeomorphism with the following properties{\rm :}
	\begin{enumerate}
	\item
		$\varphi \in W^1_{q, \loc}(\Omega)$,
		$n-1 \leq q \leq \infty${\rm ;}
	\item
		the mapping
		$\varphi$
		has finite codistortion{\rm ;}
	\item
		$\mathcal{K}_{\varphi, p} \in L_{\varrho}(\Omega)$,
		where 
		$\frac{1}{\varrho} = \frac{n-1}{q} - \frac{n-1}{p}$,
		$n-1 \leq q \leq p \leq \infty$
		$($$\varrho = \infty$
		for 
		$q = p$$)$.
	\end{enumerate}

	Then the inverse homeomorphism
	$\varphi^{-1}$
	has the following properties{\rm:}
	\begin{enumerate}
	\item
		$\varphi^{-1} \in W^1_{p', \loc}(\Omega')$,
		where
		$p' = \frac{p}{p-n+1}$,
		$($$p' = 1$
		for 
		$p=\infty$$)${\rm;}
	\item
		$\varphi^{-1}$
		has finite distortion 
		$($$J(y,\varphi^{-1}) > 0$ 
		a.e.\
		for
		$n \leq q$$)${\rm;}
	\item
		$K_{\varphi^{-1}, q'} \in L_{\varrho}(\Omega')$,
		where 
		$q' = \frac{q}{q-n+1}$
		$($$q' = \infty$
		for
		$q = n-1$$)$.
	\end{enumerate}

	Moreover,
	\begin{equation*}
		\|K_{\varphi^{-1},q'}(\cdot) \mid L_{\varrho}(\Omega')\| =
		\|\mathcal{K}_{\varphi,p}(\cdot) \mid L_{\varrho}(\Omega)\|.
	\end{equation*}
\end{theorem}

\begin{remark}\label{rem:Vod3-4}
	If replace the condition 2 on ``the mapping
	$\varphi$
	has finite distortion''
	and the condition 3 on one with the outer distortion operator function:
	``$K_{\varphi, p} \in L_{\varkappa}(\Omega)$
		where 
		$\frac{1}{\varkappa} = \frac{1}{q} - \frac{1}{p}$,
		$n-1 \leq q \leq p \leq \infty$
		$($$\varkappa = \infty$
		for 
		$q = p$$)$''.
	Then the conclusion of this theorem is valid with the next estimate
	\begin{equation*}
		\|K_{\varphi^{-1},q'}(\cdot) \mid L_{\varrho}(\Omega')\| \leq
		\|K_{\varphi,p}(\cdot) \mid L_{\varkappa}(\Omega)\|^{n-1}
	\end{equation*}
	(see \cite[Theorem 4]{Vod2012}).
\end{remark}

\begin{theorem}[\hspace{-.3pt}\cite{Vod2012, VodUhl1998, VodUhl2002}] \label{thm:bnd_comp_op}
	A~homeomorphism
	$\varphi\colon\Omega\to\Omega'$
	induces a~bounded composition operator 
	\begin{equation*}
		\varphi^*\colon L^1_p(\Omega')\to L^1_q(\Omega),
		\quad
		1\leq q \leq p <\infty,
	\end{equation*}
	where
	$\varphi^*(f)=f\circ\varphi$
	for
	$f\in L^1_p(\Omega')$,
	if and only if%
	\footnote{
		Necessity is proved in~\cite{VodUhl1998, VodUhl2002} 
		(see also earlier work~\cite{Ukh1993}), 
		and sufficiency,
		in Theorem~6 of~\cite{Vod2012}.
	}
	the following conditions hold{\rm:}
	\begin{enumerate}
	\item
		$\varphi \in W^1_{q, \loc}(\Omega)${\rm;}
	\item
		the mapping~%
		$\varphi$
		has finite distortion{\rm;}
	\item
		$K_{\varphi,p}(\cdot) \in L_{\varkappa}(\Omega)$,
		where
		$\frac{1}{\varkappa}=\frac{1}{q}-\frac{1}{p}$, $1\leq q \leq p <\infty$
		$($and
		$\varkappa=\infty$
		for
		$q=p$$)$.
	\end{enumerate}

	Moreover,
	\begin{equation*}
		\|\varphi^*\|\leq\|K_{\varphi,p}(\cdot) \mid L_{\varkappa}(\Omega)\|
		\leq C \|\varphi^*\|
	\end{equation*}
	for some constant~%
	$C$.
\end{theorem}

\begin{theorem}[{\hspace{-.3pt}\cite[Theorem 6]{Vod2012}}] 
\label{thm:Vod6}
	Assume that
	a~homeomorphism
	$\varphi\colon\Omega\to\Omega'$
	induces a~bounded composition operator
	$\varphi^*\colon L^1_p(\Omega') \to L^1_q(\Omega)$
	for
	$n-1\leq q \leq p \leq \infty$,
	where
	$\varphi^*(f)=f\circ\varphi$
	for
	$f\in L^1_p(\Omega')$
	$($and in the case 
	$p=\infty$
	the mapping 
	$\varphi$
	has finite codistortion$)$.

	Then the inverse mapping
	$\varphi^{-1}$
	induces a~bounded composition operator
	$\varphi^{-1*}\colon L^1_{q'}(\Omega) \to L^1_{p'}(\Omega')$,
	where
	$q'=\frac{q}{q-n+1}$
	and
	$p'=\frac{p}{p-n+1}$,
	and has finite distortion.

	Moreover,
	\begin{equation*}
		\|\varphi^{-1*}\|\leq\|K_{\varphi^{-1},q'}(\cdot) \mid L_{\rho}(\Omega')\|
		\leq
		\|K_{\varphi,p}(\cdot) \mid L_{\varkappa}(\Omega)\|^{n-1},
		\text{ where } 
		\frac{1}{\rho}=\frac{1}{p'}-\frac{1}{q'}.
	\end{equation*}
\end{theorem}

Recall that a differential 
$(n-1)$-form 
$\omega$
on 
$\Omega'$
is defined as
\begin{equation*}
	\omega(y) = \sum\limits_{k=1}^{n} a_k(y) \, dy_1 \wedge dy_2 \wedge \ldots \wedge
	\widehat{dy_k}
	\wedge \ldots \wedge dy_n. 
\end{equation*}
A form 
$\omega$,
with measurable coefficients
$a_k$,
belongs to 
$\mathcal{L}_p (\Omega', \Lambda^{n-1})$
if
\begin{equation*}
	\|\omega \mid \mathcal{L}_p (\Omega', \Lambda^{n-1})\| = 
	\bigg( \int\limits_\Omega \bigg(\sum\limits_{k=1}^{n} a^2_k(y)\bigg)^{p/2} \,dy\bigg)^{1/p}
	< \infty.
\end{equation*}

Let 
$f = (f_1, \ldots, f_n) \colon \Omega \to \Omega'$
belongs to 
$W^1_{q(n-1), \loc} (\Omega)$
and 
$\omega$
be a smooth
$n-1$-form. 
Then the pullback
$\varphi^*\omega$ 
can be written as
\begin{equation*}
	f^*\omega(x) = \sum\limits_{k=1}^{n} a_k(f(x)) \, 
	df_1 \wedge \ldots \wedge
	\widehat{df_k}
	\wedge \ldots \wedge df_n.
\end{equation*}

For any 
$\omega \in \mathcal{L}_p (\Omega', \Lambda^{n-1})$
the pullback operator
$\tilde{f}^* \omega(x)$ 
is defined by continuity \cite[Corollary 1.1]{Vod2010}:
\begin{equation}\label{pullback_operator}
	\tilde{f}^* \omega(x) =
	\begin{cases}
		f^* \omega(x), \qquad & \text{if } x \in \Omega \setminus (Z \cup \Sigma),\\
		0, & \text{otherwise}.
	\end{cases}
\end{equation}

As consequence of \cite[Theorem 1.1]{Vod2010} we can obtain 

\begin{theorem}\label{thm:bounded_pullback}
	A homeomorphism
	$f\colon\Omega\to \Omega'$
	induces a bounded pullback operator
	$\tilde{f}^* \colon \mathcal{L}_p(\Omega', \Lambda^{n-1}) \to \mathcal{L}_q(\Omega, \Lambda^{n-1})$,
	$1 \leq q \leq p \leq \infty$,
	if and only if{\rm:}
	\begin{enumerate}
		\item
			$f\colon\Omega\to \Omega'$
			has finite codistortion{\rm;}
		\item
			$\mathcal{K}_{f,p(n-1)} \in L_\varkappa (\Omega)$
			where 
			$\frac{1}{\varkappa} = \frac{1}{q} - \frac{1}{p}$.
	\end{enumerate}
	Moreover,
	the norm of the operator
	$\tilde{f}^*$
	is comparable with
	$\|\mathcal{K}_{f,p(n-1)} \mid L_\varkappa (\Omega) \|$.
\end{theorem}

\begin{theorem}[\hspace{-.3pt}{\cite[Theorem 1.3]{Vod2010}}]
\label{thm:pullback_inverse}
	Assume that
	a~homeomorphism
	$\varphi\colon\Omega\to\Omega'$
	belongs to
	$W^1_{n-1,\loc}(\Omega)$
	and
	induces a~bounded pullback operator
	$\widetilde{\varphi}^*\colon \mathcal{L}_p(\Omega',\Lambda^{n-1}) \to \mathcal{L}_q(\Omega, \Lambda^{n-1})$,
	for
	$1 \leq q \leq p \leq \infty$.
	Then the inverse mapping
	$\varphi^{-1} \in W^1_{1,\loc}(\Omega)$
	induces a~bounded pullback operator
	$\widetilde{\varphi}^{-1}{}^*\colon \mathcal{L}_{q'}(\Omega,\Lambda^1) \to \mathcal{L}_{p'}(\Omega, \Lambda^1)$,
	where
	$q'=\frac{q}{q-1}$
	and
	$p'=\frac{p}{p-1}$.
	Moreover,
	the norm of the operator
	$\widetilde{\varphi}^{-1}{}^*$
	is comparable with the norm of
	$\widetilde{\varphi}^*$.
\end{theorem}

\section{Almost-everywhere injectivity}\label{sec:ae_inj}

It is well known that the limit of homeomorphisms need not be 
homeomorphism or even an injective mapping.
It is illustrated by the simple example of mappings
$\varphi_k(x) = |x|^{k-1} x$
on the punctured unit ball. 
Here we have
the limit mapping
$\varphi_0 (x) \equiv 0$
and injectivity is lost.

Recall that
a~mapping
$\varphi\colon \Omega \to \mathbb{R}^n$
is called {\it injective almost everywhere}
whenever there exists a~negligible set~%
$S$
outside which~%
$\varphi$
is injective.

The sequence of homeomorphisms
$\varphi_k = (\varphi_{k,1}, \varphi_{k,2})\colon [-1,1]^2 \to [-1,1]^2$
of the class 
$W^1_2 ([-1,1]^2)$
with integrable distortion,
such that 
\begin{equation*}
	\begin{aligned}
		& \varphi_{k,1} (x_1,x_2) = 
		\begin{cases}
			2 x_1 \xi_{k} (x_2)  & \text{if } x_1\in[0, \frac{1}{2}], \\
			2(1 - \xi_{k} (x_2))x_1 - (1 - 2 \xi_{k} (x_2)) & \text{if } x_1\in(\frac{1}{2},1],
		\end{cases}\\
		& \varphi_{k,1}( -x_1, x_2) = - \varphi_{k,1}( x_1, x_2), \qquad
		\varphi_{k,1}( x_1, -x_2) = \varphi_{k,1}( x_1, x_2),
	\end{aligned}	
\end{equation*}
and
$\varphi_{k,2} (x_1,x_2) = x_2$,
with
$\xi_{k} (t) = \frac{1+(k-1) t}{2k}$,
shows that injectivity almost everywhere can be lost either.

\begin{figure}[h]
	\includegraphics[width=1\linewidth]{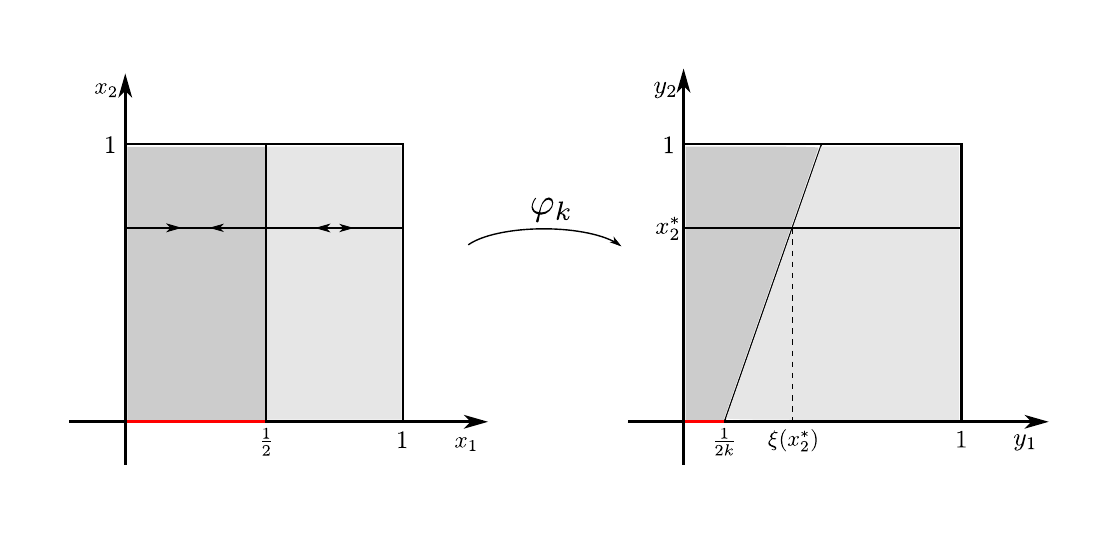}
	\caption{The sequence of homeomorphisms with an almost everywhere injective limit}
\end{figure}

\newpage 

\begin{theorem}\label{thm:ae_injectivity}
	Let
	$\Omega$, 
	$\Omega' \subset \R^n$
	be bounded domains with Lipschitz boundaries.
	Consider
	a~sequence of homeomorphisms
	$\varphi_k$,
	which maps
	$\Omega$
	onto
	$\Omega'$,
	with
	$\varphi_k \in W^1_{n-1, \loc}(\Omega)$,
	and
	$J(x,\varphi_k) \geq 0$
	a.e.,
	such that{\rm:}
	\begin{enumerate}
	\item
		$\varphi_k \rightarrow \varphi_0$
		weakly in~%
		$W^1_{n-1,\loc}(\Omega)$
		with 
		$J(x,\varphi_0) \geq 0$
		a.e.~in
		$\Omega${\rm;}

	\item
		every mapping
		$\varphi_k$
		induces a~bounded pullback operator
		$\widetilde{\varphi}_k^* \colon \mathcal{L}_{\frac{n}{n-1}}(\Omega', \Lambda^{n-1}) \to \mathcal{L}_{\frac{r}{n-1}}(\Omega, \Lambda^{n-1})$
		for some
		$n-1 \leq r \leq n${\rm;}

	\item
		the norms of the operators
		$\| \widetilde{\varphi}_k^* \|$
		are totally bounded.
	\end{enumerate}
	Then the mapping
	$\varphi_0$
	is injective almost everywhere.
\end{theorem}

By Theorem~\ref{thm:bounded_pullback} conditions 2 and 3 of Theorem~\ref{thm:ae_injectivity} can be replaced by totally boundedness of inner distortion operator functions 
$\K_{\varphi_k,n}$ 
in
$L_{\varrho}$
with
$\varrho = \frac{rn}{(n-1)(n-r)} \geq n$.

\begin{corollary}\label{cor:ae_injectivity_1}
	Let
	$\Omega$, 
	$\Omega' \subset \R^n$
	be bounded domains with Lipschitz boundaries.
	Consider
	a~sequence of homeomorphisms of finite distortion
	$\varphi_k$,
	which maps
	$\Omega$
	onto
	$\Omega'$,
	with
	$\varphi_k \in W^1_{n-1, \loc}(\Omega)$,
	and
	$J(x,\varphi_k) \geq 0$
	a.e.,
	such that{\rm:}
	\begin{enumerate}
	\item
		$\varphi_k \rightarrow \varphi_0$
		weakly in~%
		$W^1_{n-1,\loc}(\Omega)$
		with 
		$J(x,\varphi_0) \geq 0$
		a.e.~in
		$\Omega${\rm;}

	\item
		the norms of inner distortion operator functions 
		$\|\K_{\varphi_k,n} \mid L_{\varrho}\|$
		are totally bounded
		for some
		$\varrho \geq n$.
	\end{enumerate}
	Then the mapping
	$\varphi_0$
	is injective almost everywhere.
\end{corollary}

Taking into account 
$\|\K_{\varphi,n} \mid L_{ns}\| = \|K_{I} \mid L_{s}\|^{1/n}$
by Remark~\ref{rem:connection_distortion}
we derive the next assertion.

\begin{corollary}\label{cor:ae_injectivity_2}
	Let
	$\Omega$, 
	$\Omega' \subset \R^n$
	be bounded domains with Lipschitz boundaries.
	Consider
	a~sequence of homeomorphisms of finite distortion
	$\varphi_k$,
	which maps
	$\Omega$
	onto
	$\Omega'$,
	with
	$\varphi_k \in W^1_{n-1, \loc}(\Omega)$,
	and
	$J(x,\varphi_k) \geq 0$
	a.e.,
	such that{\rm:}
	\begin{enumerate}
	\item
		$\varphi_k \rightarrow \varphi_0$
		weakly in~%
		$W^1_{n-1,\loc}(\Omega)$
		with 
		$J(x,\varphi_0) \geq 0$
		a.e.~in
		$\Omega${\rm;}

	\item
		the norms of inner distortion functions 
		$\|K_{I} \mid L_{ns}\|$
		are totally bounded
		for some
		$s \geq 1$%
		\footnote{
			The exponent
			$r$ 
			from Theorem~\ref{thm:ae_injectivity}
			can be expressed as 
			$r=\frac{n (n-1) s}{ns + 1 - s} \geq n-1$
		}.
	\end{enumerate}
	Then the mapping
	$\varphi_0$
	is injective almost everywhere.
\end{corollary}

\begin{remark}\label{rem:pullback} 
	As it will be clear from the subsequent, the theorem is valid 
	provided that composition operators
	$\varphi_k^*\colon {L}^1_{n}(\Omega') \to {L}^1_{\rho}(\Omega)$,
	$1 \leq \rho < n$,
	and 
	$\psi_k^*\colon {L}^1_{r'}(\Omega) \to {L}^1_{n}(\Omega')$,
	$n \leq r' \leq \infty$,
	are bounded.
	We combine both conditions in boundedness of pullback operators
	$\widetilde{\varphi}_k^* \colon \mathcal{L}_{\frac{n}{n-1}}(\Omega', \Lambda^{n-1}) \to \mathcal{L}_{\frac{r}{n-1}}(\Omega, \Lambda^{n-1})$,
	$n-1 \leq r \leq n$.
	Indeed,
	if a homeomorphism
	$\varphi$
	induces a~bounded pullback operator
	$\widetilde{\varphi}^*\colon \mathcal{L}_{n/(n-1)}(\Omega',\Lambda^{n-1}) \to 
	\mathcal{L}_{r/(n-1)}(\Omega,\Lambda^{n-1})$
	then 
	by Theorem~{\ref{thm:pullback_inverse}} 
	the inverse mapping
	$\psi = \varphi^{-1}$
	has finite distortion and 
	induces a~bounded pullback operator
	$\widetilde\psi^*\colon \mathcal{L}_{r'}(\Omega,\Lambda^1) \to \mathcal{L}_{n}(\Omega',\Lambda^1)$
	for
	$r'=\frac{r}{r-n+1} \geq n$.
	Moreover, 
	$\|\widetilde\psi^*\| \sim \|\widetilde\varphi^*\|$%
		\footnote{
		$a \sim b$
		means there exist constants
		$C_1$,
		$C_2 > 0$,
		such that
		$C_1 a \leq b \leq C_2 a$
		}.
	As there is a case of $1$-forms, it is the same as boundedness of composition operator
	$\psi^*\colon {L}^1_{r'}(\Omega) \to {L}^1_{n}(\Omega')$
	and 
	$
	\|\psi^*\| = \|\widetilde\psi^*\|
	$.

	Further, in accordance with Theorem~\ref{thm:Vod6}
	an inverse homeomorphism 
	$\varphi = \psi^{-1}$
	has finite distortion and
	induces a~bounded composition operator
	$\varphi^*\colon {L}^1_{n}(\Omega') \to {L}^1_{\rho}(\Omega)$
	for 
	$\rho = \frac{r}{(n-1)^2 - r (n-2)} \geq 1$
	and 
	$\|\varphi^*\|  
	\sim\|\psi^*\|^{n-1}
	\sim\|\widetilde\varphi^*\|^{n-1}$.
\end{remark}

With this background the first thing we have to do is to verify that the limit mapping
$\varphi_0$
induces a bounded composition operator
$\varphi_0^*\colon L^1_n(\Omega')\cap {\Lip}(\Omega')
\to L^1_{\rho}(\Omega)$.

\begin{lemma} \label{lem:bounded_comp}
	If conditions of Theorem~\ref{thm:ae_injectivity}
	are fulfilled,
	then
	the mapping
	$\varphi_0$
	induces a~bounded composition operator
	$\varphi_0^*\colon L^1_n(\Omega')\cap {\Lip}(\Omega')
	\to L^1_{\rho}(\Omega)$, 
	$\rho = \frac{r}{(n-1)^2 - r (n-2)} \geq 1$.
\end{lemma}

\begin{proof}
Consider
$u\in L^1_{n}(\Omega')\cap {\Lip}(\Omega')$.
Since 
$\|\varphi_k^*\| \leq C$
by Remark~\ref{rem:pullback},
the sequence
$w_k=\varphi_k^*u=u\circ \varphi_k$
is bounded in
$L^1_{\rho}(\Omega)$.
Using the Poincar\'e inequality and a~compact embedding of Sobolev spaces
(see \cite[Theorem 6.2, 6.30]{Adams1975} for instance),
we obtain a~subsequence with
$w_k\rightarrow w_0$
in
$L_t(\Omega)$
where
$1 < t < \frac{n\rho}{n - \rho}$.
From this sequence,
in turn,
we can extract a~subsequence which converges almost everywhere in~%
$\Omega$.
The same arguments ensure that 
$\varphi_k \to \varphi_0$ a.e.
Then
$w_0(x)=u\circ\varphi_0(x)$
for almost all
$x\in\Omega$.

On the other hand,
since
$w_k$
converges weakly to
$w_0$
in
$L^1_{\rho} (\Omega)$,
we have 
\begin{multline*}
	\|u\circ \varphi_0 \mid L^1_{\rho}(\Omega)\|=\|w_0 \mid L^1_{\rho}(\Omega)\|\leq 
	\varliminf \limits_{k\rightarrow\infty} \|w_k \mid L^1_{\rho}(\Omega)\| 
	\\ = 
	\varliminf \limits_{k\rightarrow\infty} \|\varphi^*_k(u) \mid L^1_{\rho}(\Omega)\|
	\leq \varliminf\limits_{k\rightarrow\infty} \|\varphi_k^*\| \cdot \|u\mid L^1_n(\Omega')\|
	\\ \leq 
	C \cdot \|u \mid L^1_n(\Omega')\|.
\end{multline*}

Thus,
$\varphi_0$
induces a~bounded composition operator
$\varphi_0^*\colon L^1_n(\Omega')\cap {\Lip}(\Omega')
\to L^1_{\rho}(\Omega)$,
and moreover,
$\|\varphi_0^*\|\leq C$.
\end{proof}

Similar we can obtain boundedness of pullback operator 
$\widetilde{\varphi}_0^* \colon \mathcal{L}_{\frac{n}{n-1}}(\Omega', \Lambda^{n-1}) \to \mathcal{L}_{\frac{r}{n-1}}(\Omega, \Lambda^{n-1})$.

\begin{lemma} \label{lem:bounded_pullback}
	If conditions of Theorem~\ref{thm:ae_injectivity}
	are fulfilled,
	then
	the mapping
	$\varphi_0$
	induces a~bounded pullback operator
	$\widetilde{\varphi}_0^* \colon \mathcal{L}_{\frac{n}{n-1}}(\Omega', \Lambda^{n-1}) \to \mathcal{L}_{\frac{r}{n-1}}(\Omega, \Lambda^{n-1})$.
\end{lemma}

Now we need to consider some regularity properties of the sequence 
$\{\varphi_k\}_{k \in \mathbb{N}}$
which meet the requirements of Theorem~\ref{thm:ae_injectivity}.

\begin{lemma}\label{lem:prop_inv}
	Let conditions of Theorem~\ref{thm:ae_injectivity}
	be fulfilled,
	define a sequence of continuous mappings
	$\psi_k \colon \Omega' \to \Omega$ 
	as
	$\psi_k=\varphi^{-1}_k$.
	Then there exists a~subsequence
	$\{\psi_{k_l}\}_{l\in\mathbb{N}}$
	and a continuous mapping 
	$\psi_0 \colon \Omega' \to \Omega$
	such that
	$\psi_{k_l} \rightarrow \psi_0$
	locally uniformly.
\end{lemma}

\begin{proof}
	Notice that
	the sequence
	$\psi_k$
	is uniformly bounded
	since the domain 
	$\Omega$
	is bounded.

	On the other hand,
	since
	$\psi_k\in W^1_{n,\loc}(\Omega')$
	(by Remark~\ref{rem:pullback} and Theorem~\ref{thm:bnd_comp_op})
	we obtain the estimate
	(corollary of~\cite[Lemma~4.1]{Mos1968})
	\begin{equation*}
		{\rm osc}(\psi_k,\: S(y',r)) \leq L \biggl(\ln
		\frac{r_0}{r}\biggr)^{-\frac{1}{n}} \Biggl(\,\int\limits_{B(y',r_0)} |D
		\psi_k(y)|^n dy \Biggl)^{\frac{1}{n}},
	\end{equation*}
	where
	$S(y',r)$
	is the sphere of radius
	$r<\frac{r_0}{2}$
	centered at~%
	$y'$
	and
	$B(y',r_0)\subset\Omega'$
	is the ball of radius
	$r_0$
	centered at~%
	$y'$.
	It follows the equicontinuity of the family of functions 
	$\{\psi_k\}_{k\in\mathbb N}$ 
	on any compact part of 
	$\Omega'$.

	H\"older's inequality,
	Theorem~\ref{thm:bnd_comp_op}
	and
	Theorem~\ref{thm:Vod6}
	yield
	\begin{align*}
		\int\limits_{B(y',r_0)}&|D \psi_k (y)|^n dy \leq 
		\int\limits_{B(y',r_0)}\frac{|D \psi_k(y)|^n}{J(y,\psi_k)^{\frac{n}{r'}}} 
		J(y,\psi_k)^{\frac{n}{r'}} \,dy  
		\\ & 
		\leq\Biggl(\,\int\limits_{\Omega'} 
		\Biggl(\frac{|D\psi_k(y)|^n}{J(y,\psi_k)^{\frac{n}{r'}}}\Biggr)^{\frac{\varrho'}{n}}
		\, dy \Biggr)^{\frac{n}{\varrho'}} \Biggl(\,\int\limits_{B(y',r_0)}
		J(y,\psi_k)^{\frac{n}{r'}\cdot\frac{\varrho'}{\varrho' - n}}
		\,dy\Biggr)^{\frac{\varrho' - n}{\varrho'}}  
		\\ &
		\leq \|K_{\psi_k,r'}(\cdot) \mid L_{\varrho'}(\Omega')\|^{n}  
		|\psi_k(B(y',r_0))|^{\frac{\varrho' - n}{\varrho'}} 
		\leq \tilde C^n |\Omega|^{\frac{\varrho' - n}{\varrho'}}, 
	\end{align*}
	where
	$r' = \frac{r}{r-n+1}$,
	$\frac{1}{\varrho'} = \frac{1}{n} - \frac{1}{r'}$,
	and since
	$\frac{n}{r'}\cdot\frac{\varrho'}{\varrho' - n} =1$.

	Thus,
	we see that
	the family
	$\{\psi_k\}_{k\in\mathbb{N}}$
	is equicontinuous and uniformly bounded. 
	By the Arzel\`a--Ascoli theorem
	there exists a~subsequence 
	$\{\psi_{k_l}\}$
	converging uniformly to a~mapping
	$\psi_0$
	as
	$k_l \rightarrow \infty$.
\end{proof}

Now we verify that
the set of points
$x\in\Omega$
with
$\varphi(x) \in \partial\Omega'$
is negligible.
The proof of this statement
is based on some properties of 
additive function
$\Phi$
defined on open bounded sets.
For proving Lemma~\ref{lem:Phi} below
we modify  
the method of proof of~\cite[Theorem 4]{VodUhl2002}.

Given a~bounded open set
$A'\subset\R^n$,
define the class of functions
$\overset{\circ}L{}^1_p (A')$
as the closure of the subspace
$C_0^\infty (A')$
in the seminorm of
$L^1_p (A')$.
In general,
a~function
$f \in \overset{\circ}L{}^1_p (A')$
is defined only on the set~%
$A'$,
but,
extending it by zero,
we may assume that
$f \in L^1_p (\R^n)$.

Let us recall that a mapping
$\Phi$ 
defined on open subsets from 
$\R^n$ 
and taking nonnegative finite values is called a
\textit{monotone} if 
$\Phi(V) \leq \Phi(U)$
for
$V \subset U$
and \textit{countably additive} function of set (see \cite{VodUhl2002}) 
if for any countable set 
$U_i \subset U \subset \R^n$, 
$i = 1$, $2$, $\dots$, $\infty$, 
of pairwise disjoint open sets the following inequality 
\begin{equation*}
	\sum\limits_{i=1}\limits^{\infty} \Phi (U_i) = \Phi \left(\bigcup\limits_{i=1}\limits^{\infty} U_i\right) 
\end{equation*}
takes place.

\begin{lemma}[cf.~Lemma~1 of~\cite{VodUhl2002}]\label{lem:Phi}
	Assume that
	the mapping
	$\varphi\colon \Omega \to \overline{\Omega'}$
	induces a~bounded composition operator
	\begin{equation*}
		\varphi^*\colon L^1_p (\Omega') \cap {\Lip}(\Omega') \to L^1_q (\Omega), 
		\quad 1 \leq q < p \leq \infty.
	\end{equation*}
	Then 
	\begin{equation*}
		\Phi(A') = \sup \limits_{f \in \overset{\circ}L{}^1_p (A') \cap {\Lip}(A')} 
		\Bigg( \frac{\|\varphi^* f \mid L^1_q(\Omega) \|}
		{\|f \mid  L^1_p (A' \cap \Omega') \|} \Bigg)^{\sigma}, \quad
		\sigma = 
		\begin{cases}
			\frac{pq}{p-q} & \quad
			\text{for } p<\infty,
			\\
			q &  \quad
			\text{for } p=\infty,
	\end{cases}
\end{equation*}
is a~bounded monotone countably additive function
defined on the open bounded sets
$A'$
with
$A' \cap \Omega' \neq \emptyset$.
\end{lemma}

\begin{remark}
	If 
	$f \in \overset{\circ}L{}^1_p (A') \cap {\Lip}(A')$
	and
	$A' \not\subset \overline\Omega'$,
	we consider a composition
	$\varphi^* f$
	where it is well defined.
\end{remark}

\begin{proof}[Proof of Lemma \ref{lem:Phi}]
It is obvious that
$\Phi(A'_1)\leq \Phi(A'_2)$
whenever
$A'_1 \subset A'_2$.

Take disjoint sets
$\{A'_i\}_{i\in \mathbb{N}}$
in 
$\Omega'$
and put
$A'_0 = \bigcup\limits_{i=1}^{\infty} A'_i$.
Consider a~function
$f_i \in \overset{\circ}L{}^1_p (A'_i) \cap {\Lip}(A_i')$
such that
the conditions
\begin{equation*}
	\| \varphi^* f_i \mid L^1_q(\Omega) \| 
	\geq \Big(\Phi(A'_i) \Big(1- \frac{\varepsilon}{2^i}\Big)\Big)^{1/\sigma}
	\| f_i \mid \overset{\circ}L{}^1_p (A'_i) \|
\end{equation*}
and 
\begin{align*}
	& \| f_i \mid \overset{\circ}L{}^1_p (A'_i) \|^p = \Phi(A'_i) \Big(1- \frac{\varepsilon}{2^i}\Big)
	\text{ for } p < \infty\\
	& (\| f_i \mid \overset{\circ}L{}^1_p (A'_i) \|^p = 1
	\text{ for } p=\infty)
\end{align*}
hold simultaneously
where
$0< \varepsilon <1$.
Putting
$f_N = \sum \limits_{i=1} \limits ^{N} f_i \in L^1_p(\Omega') \cap {\Lip}(\Omega')$,
and applying H\"older's inequality in
the case of equality%
\footnote{Let us remind that for 
	$a_i$,
	$b_i \geq 0$,
	$\frac{1}{k} + \frac{1}{k'} = 1$,
	$
		\left|\sum a_i b_i\right| = \big(\sum a_i^k \big)^{1/k} \big(\sum b_i^{k'}\big)^{1/{k'}}
	$
	if and only if 
	$a_i^k$
	and 
	$b_i^{k'}$
	are proportional.},
we obtain
\begin{multline*}
	\|\varphi^* f_N \mid L^1_q(\Omega) \| \geq 
	\bigg( \sum \limits_{i=1}\limits^{N}
	\Big(\Phi(A'_i) \Big(1- \frac{\varepsilon}{2^i}\Big)\Big)^{\frac{q}{\sigma}}
	\big\| f_i | \overset{\circ}L{}^1_p (A'_i) \big\| ^q\bigg)^{\frac{1}{q}}
	\\ = 
	\bigg( \sum \limits_{i=1}\limits^{N}
	\Phi(A'_i) \Big(1- \frac{\varepsilon}{2^i}\Big)\bigg)^{\frac{1}{\sigma}}
	\bigg\| f_N \mid \overset{\circ}L{}^1_p \Big(\bigcup\limits_{i=1}\limits^{N} A'_i\Big) 
	\bigg\|
	\\ \geq 
	\bigg( \sum \limits_{i=1}\limits^{N}
	\Phi(A'_i) - \varepsilon \Phi(A'_0) \bigg)^{\frac{1}{\sigma}}
	\bigg\| f_N \mid \overset{\circ}L{}^1_p \Big(\bigcup\limits_{i=1}\limits^{N} A'_i\Big) \bigg\|,
\end{multline*}
since the sets
$A_i$,
on which the functions
$\nabla \varphi^* f_i$
are nonvanishing,
are disjoint.
This implies that
\begin{equation*}
	\Phi(A'_0)^{\frac{1}{\sigma}} \geq \sup 
	\frac{\| \varphi^* f_N  \mid L^1_p (\Omega) \|}
	{\Big\| f_N \mid \overset{\circ}L{}^1_p \Big(\bigcup\limits_{i=1}\limits^{N} A'_i\Big) 
	\Big\|}
	\\ \geq 
	\bigg( \sum \limits_{i=1}\limits^{N}
	\Phi(A'_i) - \varepsilon \Phi(A'_0) \bigg)^{\frac{1}{\sigma}},
\end{equation*}
where we take the sharp upper bound over all functions
\begin{equation*}
	f_N \in \overset{\circ}L{}^1_p 
	\Big(\bigcup\limits_{i=1}\limits^{N} A'_i\Big) \cap {\Lip}
	\Big(\bigcup\limits_{i=1}\limits^{N} A'_i\Big),
	\quad
	f_N = \sum \limits_{i=1} \limits ^{N} f_i,
\end{equation*}
and
$f_i$
are of the form indicated above.
Since
$N$~%
and~%
$\varepsilon$
are arbitrary,
\begin{equation*}
	\sum\limits_{i=1}\limits^{\infty} \Phi(A'_i) \leq \Phi 
	\Big(\bigcup\limits_{i=1}\limits^{\infty} A'_i\Big).
\end{equation*}

We can verify the inverse inequality directly
by using the definition of~%
$\Phi$.
\end{proof}

For estimating  
$\Phi$
through multiplicity of covering,
we need the following corollary to the Bezikovich theorem
(see \cite[Theorem~1.1]{Gus1978} for instance).

\begin{lemma} \label{lem:bez}
	For every open set
	$U \subset \R^n$
	with
	$U\neq \R^n$,
	there exists a~countable family
	$\mathcal{B}=\{B_j\}$
	of balls such that
	\begin{enumerate}
	\item 
		$\bigcup\limits_{j} B_j = U${\rm;}

	\item
		if
		$B_j=B_j(x_j,r_j) \in \mathcal{B}$
		then
		${\rm dist} (x_j, \partial U) = 12 r_j${\rm;}

	\item
		the families
		$\mathcal{B} = \{B_j\}$
		and
		$2 \mathcal{B} = \{2B_j\}$,
		where the symbol
		$2B$
		stands for the ball of doubled radius centered at the same point,
		constitute a covering of finite multiplicity %
		of~%
		$U${\rm;}

	\item
		if the balls
		$2 B_j = B_j(x_j, 2 r_j)$, 
		$j=1$,
		$2$,
		intersect
		then
		$\frac{5}{7} r_1 \leq r_2 \leq \frac{7}{5} r_1${\rm;}

	\item
		we can subdivide the family
		$\{2B_j\}$
		into finitely many tuples
		so that
		in each tuple the balls are disjoint
		and the number of tuples depends only on the dimension~%
		$n$.
	\end{enumerate}
\end{lemma}

\begin{lemma}
	Take a~monotone countably additive function~%
	$\Phi$
	defined on the~bounded open sets
	$A'$
	with
	$A' \cap \Omega' \neq \emptyset$.
	For every set 
	$A'$
	there exists a~sequence of balls
	$\{B_j\}_{j\in\mathbb{N}}$
	such that
	\begin{enumerate}
	\item
		the families of
		$\{B_j\}_{j\in\mathbb{N}}$
		and
		$\{2 B_j\}_{j\in\mathbb{N}}$
		constitute a covering of finite multiplicity of~%
		$U${\rm;}
	
	\item
		$\sum\limits_{j=1}\limits^{\infty} \Phi(2 B_j) \leq \zeta_n \Phi (U)$
		where the constant
		$\zeta_n$
		depends only on the dimension~%
		$n$.
	\end{enumerate}
\end{lemma}

\begin{proof}
In accordance with Lemma~\ref{lem:bez},
construct two sequences
$\{B_j\}_{j\in\mathbb{N}}$
and
$\{2 B_j\}_{j\in\mathbb{N}}$
of balls
and subdivide the latter
into
$\zeta_n$
subfamilies
$\{2 B_{1j}\}_{j\in\mathbb{N}}, \dots, \{2 B_{\zeta_n j}\}_{j\in\mathbb{N}}$
so that
in each tuple the balls are disjoint:
$2 B_{ki} \cap 2 B_{kj} = \emptyset$
for
$i \neq j$
and
$k = 1, \dots, \zeta_n$.
Consequently,
\begin{equation*}
	\sum\limits_{j=1}\limits^{\infty} \Phi(2 B_j) = \sum\limits_{k=1}
	\limits^{\zeta_n}
	\sum\limits_{j=1}\limits^{\infty} \Phi(2 B_{kj}) \leq
	\sum\limits_{k=1}\limits^{\zeta_n} \Phi(U) = 
	\zeta_n \Phi(U).
\end{equation*}
\end{proof}

Mappings inducing a~bounded composition operator is known to satisfy 
the Luzin $\N^{-1}$-property 
{\cite[Theorem 4]{VodUhl2002}}.

\begin{theorem}[\hspace{-.3pt}{\cite[Theorem 4]{VodUhl2002}}]\label{thm4}
	Take two open sets
	$\Omega$~%
	and~%
	$\Omega'$
	in
	$\R^n$
	with
	$n\geq 1$.
	If a~measurable mapping
	$\varphi\colon \Omega \to \Omega'$
	induces  a~bounded composition operator
	\begin{equation*}
		\varphi^*\colon L^1_p(\Omega') \cap C^\infty(\Omega')\to L^1_q (\Omega), 
		\quad 1\leq q \leq p \leq n,
	\end{equation*}
	then~%
	$\varphi$
	has the Luzin $\N^{-1}$-property,
	i.e.\ 
	$|\varphi^{-1} (A)| = 0$
	if 
	$|A| = 0$,
	$A \subset \Omega'$.
\end{theorem}

\begin{remark}
	Theorem~4 of~\cite{VodUhl2002} is stated for a~mapping
	$\varphi\colon \Omega \to \Omega'$
	generating a~bounded composition operator
	$\varphi^*\colon L^1_p(\Omega') \to L^1_q (\Omega)$
	with
	$1\leq q \leq p \leq n$.
	Observe that
	only smooth test functions are used in its proof,
	which therefore
	also justifies Theorem~\ref{thm4}.
\end{remark}

Here we obtain the next generalization of Theorem~\ref{thm4}.

\begin{lemma}\label{lem:N-1}
	If a~measurable 
	mapping  
	$\varphi\colon \Omega \to \overline{\Omega'}$
	induces a~bounded composition operator
	\begin{equation*}
		\varphi^*\colon L^1_p (\Omega') \cap {\Lip}(\Omega') \to 
		L^1_q (\Omega), 
		\quad 1 \leq q \leq p \leq n,
	\end{equation*}
	then
	$|\varphi^{-1}(E)| = 0$
	if
	$|E| = 0$,
	$E \subset \overline{\Omega'}$.
\end{lemma}

\begin{proof}
If 
$E \subset \Omega'$
then the statement of the theorem follows by
Theorem~\ref{thm4}.
Consider the cut-off 
$\eta \in C_0^{\infty}(\R^n)$
equal to~1 on
$B(0,1)$
and vanishing outside
$B(0,2)$.
By Lemma~\ref{lem:Phi}
the function
$f (y) = \eta \big( \frac{y-y_0}{r} \big)$
satisfies 
\begin{equation*}
	\| \varphi^* f \mid L^1_q (\Omega) \| 
	\leq C_1 \Phi(2B)^{\frac{1}{\sigma}} |B|^{\frac{1}{p}- \frac{1}{n}},
\end{equation*} 
where
$B \cap \Omega' \neq \emptyset$
(let 
$\Phi(2B)^{\frac{1}{\sigma}} = \|\varphi^*\|$
for any ball 
$B$
if 
$p = q$).
Take a~set 
$E \subset {\partial \Omega}'$
with
$|E| = 0$.
Since~%
$\varphi$
is a~mapping with 
finite distortion \cite{VodUhl2002},
$\varphi^{-1}(E) \neq \Omega$
(otherwise,
$J(x,\varphi)=0$
and,
consequently,
$D\varphi(x) = 0$,
that is,~%
$\varphi$
is a~constant mapping).
Hence,
there is a~cube
$Q \subset \Omega$
such that
$2Q \subset \Omega$
and
$| Q \setminus \varphi^{-1} (E) | >0$
(here
$2Q$
is a~cube with the same center as~%
$Q$
and the edges stretched by a~factor of two compared to~%
$Q$).
Since~%
$\varphi$
is a~measurable mapping,
by Luzin's theorem  there is a~compact set
$T \subset Q \setminus \varphi^{-1} (E)$
of positive measure such that
$\varphi\colon T \to \overline{\Omega'}$
is continuous.
Then,
the image
$\varphi(T) \subset \overline{\Omega'}$
is compact and
$\varphi(T) \cap E = \emptyset$.
Consider an~open set 
$U \supset E$
with
$\varphi(T)\cap U = \emptyset$
and
$U\cap \Omega' \neq \emptyset$.
Choose a~tuple
$\{B(y_i, r_i)\}_{i\in\mathbb{N}}$
of balls in accordance with Lemma~\ref{lem:bez}:
$\{B(y_i, r_i)\}_{i\in\mathbb{N}}$
and
$\{B(y_i, 2 r_i)\}_{i\in\mathbb{N}}$
are coverings of~%
$U$,
and the multiplicity of the covering
$\{B(y_i, 2 r_i)\}_{i\in\mathbb{N}}$
is finite
($B(y_i, 2 r_i) \subset U$
for all~%
$i \in \mathbb{N}$).
Then the function
$f_i$
associated to the ball
$B(y_i, r_i)$
enjoys
$\varphi^*f_i = 1$
on 
$\varphi^{-1}(B(y_i, r_i))$
and
$\varphi^* f = 0$
outside 
$\varphi^{-1}(B(y_i, 2 r_i))$,
in~particular 
$\varphi^* f_i = 0$
on~%
$T$.
In addition,
we have the estimate
\begin{equation*}
	\| \varphi^* f_i \mid L^1_q (2Q) \| \leq \| \varphi^* f_i \mid L^1_q (\Omega)  \| 
	\leq 
	C_1 \Phi(B(y_i, 2 r_i))^{\frac{1}{\sigma}} |B(y_i, r_i)|^{\frac{1}{p}- \frac{1}{n}}.
\end{equation*} 
By the Poincar\'e inequality
(see~\cite{Maz2011} for instance),
for every function
$g \in W^1_{q,\loc} (Q)$
with
$q<n$
vanishing on~%
$T$,
we have
\begin{equation*}
	\bigg( \int\limits_Q | g |^{q^*} \,dx \bigg)^{1/q^*} 
	\leq 
	C_2 l(Q)^{n/q^*}\bigg( \int\limits_{2Q} |\nabla g |^{q} \,dx \bigg)^{1/q}
\end{equation*} 
where
$q^* = \frac{nq}{n-q}$
and
$l(Q)$
is the edge length of~%
$Q$.

Applying the Poincar\'e inequality to the function
$\varphi^* f_i$
and using the last two estimates,
we obtain
\begin{equation*}
	|\varphi^{-1}(B(y_i, r_i)) \cap Q|^{\frac{1}{q} - \frac{1}{n}}  
	\leq
	C_3 \Phi(B(y_i, 2 r_i))^{\frac{1}{\sigma}} 
	|B(y_i, r_i)|^{\frac{1}{p}- \frac{1}{n}}.
\end{equation*}
Note, that the constant 
$C_3$
can depend on the cube
$Q$.
In turn,
H\"older's inequality guarantees that %
\begin{multline*}
	\Bigg(\sum\limits_{i=1}\limits^{\infty}|\varphi^{-1}(B(y_i, r_i)) \cap Q|\Bigg)
	^{\frac{1}{q} - \frac{1}{n}}
	\\ \leq
	C_3 \Bigg(\sum\limits_{i=1}\limits^{\infty}\Phi(B(y_i, 2 r_i))\Bigg)^{\frac{1}{\sigma}}
	\Bigg(\sum\limits_{i=1}\limits^{\infty}|B(y_i, r_i)|\Bigg)^{\frac{1}{p}- \frac{1}{n}}.
\end{multline*} 
As the open set~%
$U$
is arbitrary,
this estimate yields
$|\varphi^{-1}(E) \cap Q|=0$.
Since the cube
$Q \subset \Omega$
is arbitrary,
it follows that
$|\varphi^{-1}(E)|=0$.
\end{proof}

The sequence
$\{\varphi_k\}_{k\in\mathbb{N}}$
converges weakly in
$W^1_{r,\loc} (\Omega)$.
Therefore,
by embedding theorem 
picking up the subsequence if necessary,
it is reputed that
$\varphi_0$
is an almost everywhere pointwise limit of the homeomorphisms
$\varphi_k\colon \Omega \to \Omega'$.
In this case
the images of some points
$x\in \Omega$
may belong to the boundary
$\partial \Omega'$.

Denote by
$S \subset\Omega$
a~negligible set
on which the convergence 
$\varphi_k(x)\rightarrow \varphi_0(x)$
as
$k \rightarrow \infty$
fails.
If
$x\in\Omega\setminus S$
with 
$\varphi(x) \in \Omega'$
then the
injectivity follows from the uniform convergence of
$\psi_k = \varphi_k^{-1}$
on~%
$\Omega'$ 
(see Lemma~\ref{lem:prop_inv})
and the identity
\begin{equation*}
	\psi_k \circ \varphi_k(x)=x, \quad x \in \Omega \setminus S,
\end{equation*}
Passing to the limit as
$k \rightarrow \infty$,
we infer that

\begin{equation*}
	\psi_0 \circ \varphi_0(x)=x,\: x\in\Omega\setminus S.
\end{equation*}
Hence,
we deduce that
if
$\varphi_0(x_1)=\varphi_0(x_2)\in  \Omega'$
for two points
$x_1,x_2\in\Omega\setminus S$
then
$x_1=x_2$.

Since for the domain
$\Omega'$
with Lipschitz boundary
we have 
${|\partial \Omega'| = 0}$,
Lemmas~\ref{lem:bounded_comp}~and~\ref{lem:N-1} imply Theorem~\ref{thm:ae_injectivity}.

Let us mention another interesting corollary of Theorem~\ref{thm4}.
Recall that 
a mapping
$f\colon \Omega \to \Omega'$
is said to be
\textit{approximative differentiable}
at
$x \in \Omega$
with approximative derivative
$Df(x)$
if there 
is a set
$A\subset\Omega$
of density one at
$x$%
\footnote{
	i.e.\
	$\lim_{r\to 0} \frac{|A \cap B(x,r)|}{|B(x,r)|} = 1$
}
such that
\begin{equation*}
	\lim\limits_{y\to x, \: y \in A} \frac{f(y) - f(x) - Df(x)(y - x)}{\|y - x\|} = 0.
\end{equation*}
It is well known that Sobolev functions are approximative differentiable a.e.
(see \cite{Fed1969, HenKos2014} for more details).

\begin{lemma}\label{lem:J>0}
	If an~almost everywhere injective mapping
	$\varphi\colon \Omega \to \Omega'$
	with
	$\varphi \in W^1_1(\Omega)$
	and
	$J(x,\varphi) \geq 0$
	a.e.~in
	$\Omega$
	has the Luzin $\N^{-1}$-property
	then
	${J(x,\varphi) > 0}$
	for almost all
	$x \in \Omega$.
\end{lemma}

\begin{proof}
	Let~%
	$E$
	be a~set outside which the mapping~%
	$\varphi$
	is approximatively differentiable
	and has the Luzin $\N^{-1}$-property. 
	Since
	$\varphi \in W^1_1(\Omega)$,
	then
	$|E| = 0$
	(see~\cite{Haj1993,Whit1951}).
	In addition,
	we may assume that
	$\{ x\in \Omega \setminus E \mid J(x, \varphi) = 0\}$
	is contained in a~Borel set 
	$Z$
	of measure zero. 
	Put
	$\sigma = \varphi(Z)$.
	By the change-of-variable formula
	{\cite[Theorem 2]{Haj1993}},
	taking the injectivity of~%
	$\varphi$
	into account,
	we obtain
	\begin{equation*}
		\int\limits_{\Omega \setminus \Sigma} \chi_{Z} (x) J (x, \varphi) \, dx
		= \int\limits_{\Omega \setminus \Sigma} (\chi_{\sigma} \circ \varphi)(x) J (x, \varphi) \, dx 
		= \int\limits_{\Omega'} \chi_{\sigma} (y) \, dy. 
	\end{equation*}
	By construction,
	the integral in the left-hand side vanishes;
	consequently,
	$|\sigma| = 0$.
	On the other hand,
	since~%
	$\varphi$
	has the Luzin $\N^{-1}$-property,
	we have
	$|Z| = 0$.
\end{proof}

\section{Elasticity}\label{sec:elasticity}

The goal of this section is to prove the existence theorem for minimizing problem  of energy functional  in the classes
$\Hom(n-1, s, M;{\overline\varphi})$
where $s\in[1,\infty]$. Our prove works for all values of parameter $s$. 
It is worth to note that at $s=1$ some results of this section look like  some statements of paper~\cite{IwaOnn2009}. 
In our proof we use different arguments, such as the
boundedness of composition operators.
It gives an opportunity to apply them to new classes of deformations.
Naturally,
the proof of our main result 
differs substantially from previous works
and is based crucially on the results and methods of~\cite{Vod2012}.

Comparison of our results with those in another papers see in Remark~\ref{rem:free_outer_dist} and Section~\ref{sec:examples}.

\subsection{Polyconvexity}\label{sec:concept_poly}

Let 
$F= [f_{ij}]_{i,j=1,\dots, n}$
be a
$(n \times n)$-matrix.
For every pair of ordered tuples
$I=(i_1,i_2,\dots i_l)$,
$1 \leq i_1 < \dots < i_l \leq n$, 
and
$J=(j_1,j_2,\dots j_l)$,
$1 \leq j_1 < \dots < j_l \leq n$,
define
$l\times l$-minor
of the matrix 
$F$
\begin{equation*}
	F_{IJ} = \left|
	\begin{array}{ccc}
		f_{i_1 j_1} & \cdots & f_{i_1 j_l} \\
		\vdots & \ddots & \vdots \\
		f_{i_l j_1} & \cdots & f_{i_l j_l}
	\end{array}
	\right|.
\end{equation*}
Notice that 
$n\times n$-minor is the determinant of
$F$.
Let 
$F_{\#}$
be an ordered list of all minors of 
$F$.
Let
$F_{\#} \in D \subset \mathbb{R}^N$
for sufficiently large $N$
($N = {2n \choose n}$),
where
$D$
be a convex set with nonnegative 
$n\times n$-minor.	
\begin{definition}[\hspace{-.3pt}\cite{Ball1977}]
	A~function
	$W\colon\mathbb{M}^{n\times n}\to \R$
	is {\it polyconvex}
	if there exists a~convex function
	$G\colon D\to \R$,
	such that
	\begin{equation*}
		G(F_{\#})=W(F).
	\end{equation*}
\end{definition}

Examples of polyconvex
but not convex functions are
\begin{equation*}
	W(F)= \det F
\end{equation*}
and
\begin{equation*}
	W(F)=\tr \Adj F^T F = \|\Adj F^T F\|^2
\end{equation*}
(see, for example, \cite{Ciar1988}).

It is known that
for a~hyperelastic material with experimentally known Lam\'e coefficients
it can be constructed a~stored-energy function of an~Ogden material
(see~\cite{Ogd1972,Ciar1988} for more details).
On the other hand, a well-known Saint-Venant--Kirchhoff material,
is not polyconvex 
\cite[Theorem 4.10]{Ciar1988}.

\subsection{Existence theorem}\label{subsec:exist}

Let 
$\Omega$, 
$\Omega'\subset \R^n$
be two bounded domains with Lipschitz boundaries.
Recall that 
a~mapping 
$G\colon\Omega \times \R^{m} \to \overline{\R}$
enjoys the~{\it Carath\'eodory conditions}
whenever 
	$G(x, \cdot)$
	is continuous on
	$\R^{m}$
	for almost all
	$x\in \Omega$;
and
	$G(\cdot, a)$
	is measurable on~%
	$\Omega$
	for all
	$a\in\R^{m}$.

Consider 
a functional 
\begin{equation*}
	I(\varphi)=\int\limits_{\Omega}W(x, D\varphi(x))\,dx,
\end{equation*}
where
$W\colon \Omega \times \M^{n\times n}\to \R$
is a~stored-energy function
with the following properties:

{\bf (a)  polyconvexity:}
	there exists a~convex function
	$G \colon \Omega \times D
	\to \R$,
	$D \subset \R^N$,
	meeting Carath\'eodory conditions
	such that
	for all
	$F\in \M^{n \times n}$,
	$\det F \geq 0$,
	the~equality
	\begin{equation*}
		G(x, F_{\#})=W(x,F)
	\end{equation*}
	holds almost everywhere in~%
	$\Omega$;

{\bf (b) coercivity:}
	there exists a constant
	$\alpha>0$
	and a~function
	$g\in L_1(\Omega)$
	such that
	\begin{equation}\label{neq:coer}
		W(x,F)\geq \alpha |F|^n + g(x)
	\end{equation}
	for almost all
	$x\in \Omega$
	and all
	$F\in \M^{n \times n}$,
	$\det F \geq 0$.

Given
constants
$p$,
$q\geq 1$,
$M > 0$
define the class of admissible deformations
\begin{multline}\label{def:Hom}
	\Hom(p,q,M)=\{
	\varphi \colon \Omega \to \Omega' \text{ is a homeomorphism with finite distortion, }
	\\
	\varphi \in W^1_1(\Omega),\:
	I(\varphi) < \infty, \:
	\: J(x,\varphi) \geq 0 
	\text{ a.e.~in } \Omega, \: 
	\\
	K_O(\cdot, \varphi) \in L_{p} (\Omega), \:
	\| K_I(\cdot,\varphi) \mid L_{q} (\Omega)\| \leq M \},
\end{multline}
where
$K_O(x,\varphi)$
and
$K_I(x,\varphi)$
are the outer and the inner distortion functions defined by \eqref{def:outer_inner_distortion}.

For these families of admissible deformations  
we have natural embeddings 
$$
	\Hom(p,q_2,M_2) \subset \Hom(p,q_1,M_1)
$$
if 
$q_1 \leq q_2$
and
$M_2 |\Omega|^{\frac{1}{q_1} - \frac{1}{q_2}} \leq M_1$.
If 
$p_1 \leq p_2$
then
$$
	\Hom(p_2,q,M) \subset \Hom(p_1,q,M)
$$
also holds.

\begin{theorem}[Existence theorem] \label{thm:main}
	Suppose that
	conditions {\bf (a)}~and~{\bf (b)}
	on the function
	$W(x,F)$
	are fulfilled and
	the set~%
	$\Hom(n-1,s,M)$
	is nonempty,
	$M > 0$,
	$s > 1$.
	Then there exists at least one homeomorphic mapping
	\begin{equation*}
		\varphi_0\in\Hom(n-1,s,M)
		\quad\text{such that}\quad 
		I(\varphi_0)=\inf\limits\{I(\varphi),\varphi\in\Hom(n-1,s,M)\}.
	\end{equation*}
\end{theorem}

If there is a homeomorphic Dirichlet data 
$\overline{\varphi}\colon \overline\Omega \to \overline\Omega'$,
$\overline{\varphi}\in W^1_n(\Omega)$,
$J(x, \overline \varphi) > 0$
a.e.\ in
$\Omega$,
$\|K_I(\cdot,\overline\varphi) \mid L_{q} (\Omega)\| \leq M$,
and
$I(\overline\varphi) < \infty$,
than we can define the classes of admissible deformations 
\begin{equation*}
	\Hom(p, q, M;{\overline\varphi})=\{
	\varphi \in \Hom(p, q, M), \: 
	\varphi|_{\partial \Omega}=\overline{\varphi}|_{\partial \Omega}
	\text{ a.e.~on } {\partial \Omega} \}.
\end{equation*}

Because of Theorem~\ref{thm:main} and compactness of the trace operator
(see \cite[Sect. 1.4.5--1.4.6]{Maz2011}
for instance)
it can be easily obtained the next existence theorem 
with respect to a Dirichlet boundary condition 
$\varphi|_{\partial \Omega}=\overline{\varphi}|_{\partial \Omega}$
a.e.\ on
${\partial \Omega}$.

\begin{corollary} \label{cor:ex_thm_dd}
	Suppose that
	conditions {\bf (a)}~and~{\bf (b)}
	on the function
	$W(x,F)$
	are fulfilled and
	the set~%
	$\Hom(n-1, s, M;{\overline\varphi})$
	is nonempty,
	$M > 0$,
 	$s \geq 1$.
	Then there exists at least one mapping
	$\varphi_0\in\Hom(n-1, s, M;{\overline\varphi})$
	such that
	\begin{equation*}
		I(\varphi_0)=\inf\{I(\varphi),\: {\varphi\in\Hom(n-1, s, M;{\overline\varphi})} \}.
	\end{equation*}
\end{corollary}

In some cases it is more convenient to consider
deformations of the same homotopy class as a given homeomorphism
$\overline\varphi$
instead of deformations with prescribed boundary values.

In this case 
we can define the next class of admissible deformations
\begin{multline*}
	\Hom(p,q,M;\overline\varphi, hom)=\{
	\varphi \in \Hom(p, q, M), \: 
	\\
	\varphi \text{ belongs to the same homotopy class as } \overline{\varphi}
	\}.
\end{multline*}

\begin{corollary} \label{cor:ex_thm_hc}
	Suppose that
	conditions {\bf (a)}~and~{\bf (b)}
	on the function
	$W(x,F)$
	are fulfilled and
	the set~%
	$\Hom(n-1, s, M;{\overline\varphi}, hom)$
	is nonempty,
	$M > 0$,
	$s \geq 1$.
	Then there exists at least one mapping
	$\varphi_0\in\Hom(n-1, s, M;{\overline\varphi}, hom)$
	such that
	\begin{equation*}
		I(\varphi_0)=\inf\{I(\varphi),\: {\varphi\in\Hom(n-1, s, M;{\overline\varphi}, hom)} \}.
	\end{equation*}
\end{corollary}

\begin{remark}\label{rem:free_homeo}
	Note that we can omit the condition that 
	$\varphi$
	is a homeomorphism in the definition of 
	$\Hom(n-1, s, M;{\overline\varphi}, hom)$
	(and $\Hom(n-1, s, M;{\overline\varphi})$)
	if 
	$s > 1$.
	Since
	$\varphi \in \Hom(n-1, s, M;{\overline\varphi}, hom)$
	belongs to 
	$W^1_n(\Omega)$,
	$K_O(\cdot,\varphi) \in L_{n-1}(\Omega)$
	and
	$K_I(\cdot,\varphi) \in L_{s}(\Omega)$,
	$s > 1$,
	the mapping 
	$\varphi$
	is continuous, open and discrete
	(Theorem~\ref{thm:Raj} and 
	\cite{Raj2010}).
	Also, it is known that
	continuous open discrete mapping
	$\varphi$,
	with the same homotopy class as a given homeomorphism
	$\overline{\varphi} \in W^1_n (\Omega)$,
	is also a homeomorphism
	of 
	$\Omega$
	onto
	$\Omega'$
	(see \cite{IwaOnn2009} for instance).
\end{remark}

\begin{remark}\label{rem:free_outer_dist}
	Added to this is the fact that if we have boundary conditions, 
	we do not need restriction on
	$K_O(x,\varphi)$ 
	(see Remark~\ref{rem:boundary_condition} for details). 
	Thereafter for
	$s \geq 1$
	instead of 
	$\Hom(n-1, s, M;{\overline\varphi})$
	and
	$\Hom(n-1, s, M;{\overline\varphi}, hom)$
	we can consider classes 
	\begin{align}\label{admissible_classes}
		\A(s, M;{\overline\varphi})= & \{
		\varphi \in \A(s, M), \: 
		\varphi|_{\partial \Omega}=\overline{\varphi}|_{\partial \Omega}
		\text{ a.e.~on } {\partial \Omega} \} \quad \text{and} 
		\\
		\A(s,M;\overline\varphi, hom)= & \{
		\varphi \in \A(s, M), \: \nonumber
		\\
		& \qquad \varphi \text{ belongs to the same homotopy class as } \overline{\varphi}
		\}, 
	\end{align}
	where
	\begin{multline*}
		\A(s,M)=\{
		\varphi \colon \Omega \to \Omega' \text{ is a homeomorphism with finite distortion, } \\
		\varphi \in W^1_1(\Omega),\:
		I(\varphi) < \infty, \:
		\: J(x,\varphi) \geq 0 
		\text{ a.e.~on } \Omega, \: 
		\\
		\| K_I(\cdot,\varphi) \mid L_{s} (\Omega)\| \leq M \}.
	\end{multline*}
	Note that, for a mapping being of the class 
	$\A(1,M)$
	we ask the same requirements as those in the paper \cite{IwaOnn2009}.
\end{remark}

\subsection{Proof of the existence theorem}\label{subsec:proof}

In this section we prove the existence of a~minimizing mapping
for the functional 
\begin{equation*}
	\overline{I}(\varphi)=I(\varphi)- \int\limits_\Omega g(x)\, dx.
\end{equation*}

Observe now that
the coercivity~\eqref{neq:coer}
of the function~%
$W$
and the corollary of the Poincar\'e inequality
(see~\cite[Theorem 6.1-8]{Ciar1988} for instance)
ensure the existence of constants
$c>0$
and
$d\in\R$
such that
\begin{equation}\label{est:lower}
	\overline{I}(\varphi) = I(\varphi)- \int\limits_\Omega g(x)\, dx
	\geq
	c\|\varphi\mid W^1_n(\Omega)\|^n + d
\end{equation}
for every mapping
$\varphi \in \Hom =\Hom(n-1,s,M)$,
where 
$\Hom$
is defined by~\eqref{def:Hom}.

Take a~minimizing sequence
$\{\varphi_k\}$
for the functional~%
$\overline{I}$.
Then
\begin{equation*}
	\lim\limits_{k\rightarrow \infty} \overline{I}(\varphi_k)
	= \inf\limits_{\varphi\in\Hom} \overline{I}(\varphi).
\end{equation*}
By~\eqref{est:lower} and the assumption
$\inf\limits_{\varphi\in\Hom} \overline{I} (\varphi)<\infty$,
the sequence
$\{\varphi_k\}_{k\in \mathbb{N}}$
is bounded in 
$W^1_n(\Omega)$.

Remind that Sobolev space 
$W^1_{n}$
has 
the ``continuity'' property of minors~--- 
rank-$l$ minors of  
$D\varphi_k$ 
are weakly converging if 
$\varphi_k$ 
belongs to 
$W^1_{p}$ 
with
$p \geq l$,
$1 \leq l < n$ 
\cite{Ball1977,Mor1966,Resh1967,Resh1982}.
In the case 
$l = n$
there is no weak convergence but something close to it
\cite[\S 4.5]{Resh1982}.
For achieving weak convergence of Jacobians,  
it is necessary to impose some 
additional conditions, 
for instance,
nonnegativity of Jacobians almost everywhere
\cite{Mul1990}.
Here it will be convenient for us  the next formulation of this assertion,
which can be found in \cite{GehrIwa1999}.

\begin{lemma}[Weak continuity of minors]\label{lem:weak_cont}
	Let
	$\Omega$
	be a~domain in
	$\R^n$
	and a sequence 
	$f_k\colon \Omega \to \mathbb{R}^n$,
	$k = 1$,
	$2, \dots$,
	converge weakly in
	$W^1_{n,\loc}(\Omega)$
	to a mapping 
	$f_0$.
	For  
	$l$-tuples 
	$1 \leq i_1 < \dots < i_l \leq n$
	and
	$1 \leq j_1 < \dots < j_l \leq n$
	the equality 
	\begin{equation}\label{weak_cont_minors}
		\lim\limits_{k\to \infty} \int\limits_{\Omega} \theta 
			\frac{\partial (f_k^{i_1},\dots,f_k^{i_l})}{\partial (x_{j_1},\dots,x_{j_l})} \, dx = 
		\int\limits_{\Omega} \theta 
			\frac{\partial (f_0^{i_1},\dots,f_0^{i_l})}{\partial (x_{j_1},\dots,x_{j_l})} \, dx
	\end{equation}
	holds
	for every 
	$\theta $
	in
	$\overset{\circ}{L}_{n/(n-l)}(\Omega)$,
	the space of functions in
	$L_{n/(n-l)}(\Omega)$ 
	with compact support in
	$\Omega$,
	and corresponding 
	$l \times l$
	minors%
	\footnote{
		i.e.\ determinants of the matrix that is formed
		by taking the elements of the original matrix 
		from the rows whose indexes are in 
		$({i_1},i_2,\dots,{i_l})$ 
		and columns whose indexes are in 
		$({j_1},j_2,\dots,{j_l})$
	}
	of
	$D f_k$
	and
	$D f_0$,
	$l=1,2,\dots,n-1$.
	
	Moreover, 
	if in addition
	$J(x, f_k) \geq 0$
	a.e.\ in
	$\Omega$,
	the equality \eqref{weak_cont_minors}
	holds for 
	$l=n$.
\end{lemma}

Hence there exists a~minimizing sequence fulfilling the conditions
\begin{equation*}
	\begin{cases}
		\varphi_k \longrightarrow \varphi_0 &
		\text{ weakly in }  W^1_{n}(\Omega),
		\\
		\Adj D\varphi_k \longrightarrow \Adj D\varphi_0 &
		\text{ weakly in }  
		L_{\frac{n}{n-1},\loc}(\Omega),
		\\
		\dots & \\
		J (\cdot,\varphi_k) \longrightarrow J (\cdot,\varphi_0) &
		\text{ weakly in } L_{1,\loc}(\Omega)
	\end{cases}
\end{equation*}
as
$k \rightarrow \infty$,
where
$\varphi_0$
guarantees the sharp lower bound 
$\overline{I} (\varphi_0)=\inf\limits_{\varphi\in\Hom} \overline{I} (\varphi)$.
It remains to verify that
$\varphi_0 \in \Hom$.
To this end,
we need the properties of mappings of~%
$\Hom$.

\begin{lemma}
	The limit mapping
	$\varphi_0$
	satisfies
	$J(\cdot,\varphi_0)\geq 0$
	a.e.\ in 
	$\Omega$.
\end{lemma}

The inequality
$J(\cdot,\varphi_0)\geq 0$
follows directly from the weak convergence of
$J(\cdot,\varphi_k)$
in
$L_{1}(K)$,
for every 
$K \Subset \Omega$.
Among other things,
we can establish the nonnegativity of the Jacobian
by using weak convergence 
(see \cite[\S{4.5}]{Resh1982}).


Now 
by Corollary~%
\ref{cor:ae_injectivity_2}
the mapping
$\varphi_0$
is almost-everywhere injective
(moreover according to the proof of Theorem~\ref{thm:ae_injectivity}, injectivety can be lost only if points go to the boundary).
Furthermore,
since 
$\varphi_0 \in W^1_n(\Omega)$
has finite distortion
(by Lemma~\ref{lem:bounded_comp} and Lemma~\ref{lem:FD}) 
and if
$K_O(\cdot,\varphi_0) \in L_{n-1}(\Omega)$
and
$K_I(\cdot,\varphi_0) \in L_{s}(\Omega)$,
$s>1$,
then the mapping
$\varphi_0$
is continuous, discrete and open
by Theorem~%
\ref{thm:Raj}.
Therefore so
$\varphi_0$
is a~homeomorphism.

Moreover, the proof of Theorem~\ref{thm:ae_injectivity} results in 
the Lusin
$\N^{-1}$-property for 
$\varphi_0$
(see Lemma~\ref{lem:N-1}).
Then, Lemma~\ref{lem:J>0}
implies
the limit mapping
$\varphi_0$
satisfies the~strict inequality
$J(x,\varphi_0) > 0$
a.e.~in~%
$\Omega$.

\begin{remark}\label{rem:boundary_condition}
	Theorem~\ref{thm:Raj}
	is not known if 
	$s = 1$. 
	However, we include the case
	$s=1$
	for classes
	$\Hom(n-1,s,M; \overline\varphi)$
	and
	$\Hom(n-1,s,M; \overline\varphi, hom)$
	($\A(s,M; \overline\varphi)$
	and
	$\A(s,M; \overline\varphi, hom)$).
	Indeed, whereas both
	$\varphi_k$ 
	and 
	$\psi_k$ 
	belong to Sobolev spaces
	$W^1_n(\Omega)$
	and 
	$W^1_n(\Omega')$,
	the same arguments as in Lemma~\ref{lem:prop_inv} ensure that
	there are a~sequence of homeomorphisms
	$\{\varphi_k\}_{k\in \mathbb{N}}$
	and 
	a~sequence of inverse homeomorphisms
	$\{\psi_k\}_{k\in \mathbb{N}}$,
	which converge locally uniformly
	to
	$\varphi_0$
	and
	$\psi_0$
	respectively.

	Then 
	$\varphi_0$
	and
	$\psi_0$
	are continuous and
	\begin{equation*}
		\psi_0 \circ \varphi_0 (x) = x, \qquad \varphi_0 \circ \psi_0 (y) = y,
	\end{equation*} 
	if 
	$\varphi_0 (x) \not \in \partial \Omega'$
	and
	$\psi_0 (y) \not \in \partial \Omega$.
	
	Since 
	$\varphi_0$
	coincides with given homeomorphism 
	$\overline \varphi$
	on the boundary
	(or is in the same homotopy class),
	$\mu(y,\Omega,\varphi_0) = 1$ 
	for 
	$y \not \in \varphi_0(\partial \Omega)$.
	Therefore for 
	$y\in\Omega'$
	there is 
	$x \in \Omega$ 
	such that 
	$\varphi_0(x) = y \in \Omega'$.
	Passing to the limit in 
	$\psi_k \circ \varphi_k (x) = x$,
	we obtain
	$\psi_0(y) = x \in \Omega$.
	Similar we obtain
	$\varphi_0(x) = y \in \Omega'$
	for 
	$x \in \Omega$.
\end{remark}

In order to make sure that 
$\varphi_0 \in \Hom$
it remains to verify 
\begin{equation*}
 	K_O(\cdot, \varphi_0) \in L_{n-1}(\Omega) \quad \text{and} \quad 
 	\| K_I(\cdot,\varphi_0) \mid L_{s} (\Omega) \| \leq M.
\end{equation*}
It follows from the semicontinuity property of distortion coefficient 
\cite{GehrIwa1999}, \cite[Theorem 8.10.1]{IwaMar2001}
(see this property under weaker assumption and some generalization in \cite{VodMol2016,VodKudr2017}).

In order to complete the proof,
it remains to verify lower semicontinuity of the functional
\begin{equation*}
	\int\limits_\Omega W(x, D\varphi_0)\,dx\leq \varliminf\limits_{k\rightarrow\infty} 
	\int\limits_\Omega W(x,D\varphi_k)\,dx,
\end{equation*}
using conventional technique for polyconvex case
(see, for example, \cite[\S 5]{Mul1990}).

\section{Examples}\label{sec:examples}

As our first example
consider an~Ogden material
with the stored-energy function
$W_1$
of the form
\begin{equation}\label{example1}
	W_1(F)=a\tr(F^T F)^{\frac{p}{2}} + b \tr\Adj (F^T F)^{\frac{q}{2}} 
	+ c(\det F)^r + d (\det F)^{-m},
\end{equation}
where
$a > 0$,
$b > 0$,
$c > 0$,
$d > 0$,
$p > 3$,
$q > 3$,
$r > 1$,
and
$m > \frac{2q}{q-3}$.
Then
$W_1(F)$
is polyconvex
and the coercivity inequality holds
\cite[Theorem 4.9-2]{Ciar1988}:
\begin{equation*}
	W_1(F)\geq \alpha\big(|F|^p+|\Adj F|^q\big)+ c(\det F)^r + d (\det F)^{-m}.
\end{equation*}
We have to solve the~minimization problem 
\begin{equation}\label{problem:example1}
	I_1(\varphi_B)=\inf \{I_1(\varphi) : \varphi\in\A_B\},
\end{equation}
where
$I_1(\varphi) = \int\limits_{\Omega} W_1 (D\varphi (x)) \, dx$
and the class of admissible deformations
$\A_B = \{\varphi\in W^1_1(\Omega), \: I_1(\varphi) < \infty, \:
	J(x,\varphi) > 0 
	\text{ a.e.~in } \Omega, 
	\varphi|_{\partial \Omega}=\overline{\varphi}|_{\partial \Omega}
	\text{ a.e.~on } \partial\Omega\}$
is defined by~\eqref{def:AB}
for a homeomorphic boundary conditions 
$\overline{\varphi}\colon \overline\Omega \to \overline{\Omega'}$,
$\overline{\varphi}\in W^1_p(\Omega)$,
$J(x, \overline \varphi) > 0$
a.e.\ in
$\Omega$
and
$I_1(\overline\varphi) < \infty$.
The result of John Ball~\cite{Ball1981} ensures that
there exists at least one solution
$\varphi_B\in\A_B$
to this problem,
which is a~homeomorphism in addition.

Denote 
$\inf\limits_{\varphi\in\A_B} I_1(\varphi) + m = M$
for any 
$m>0$
and consider a class, defined by~\eqref{admissible_classes},
\begin{multline*}
	\A(s,M;\overline\varphi)=\{
	\varphi \colon \Omega \to \Omega' \text{ is a homeomorphism with finite distortion, } 
	\\
	\varphi \in W^1_1(\Omega),\:
	I_1(\varphi) < \infty, \:
	\: J(x,\varphi) \geq 0 
	\text{ a.e.~in } \Omega, \: 
	\\
	\| K_I(\cdot,\varphi) \mid L_{s} (\Omega)\| \leq M, \: 
	\varphi|_{\partial \Omega}=\overline{\varphi}|_{\partial \Omega}
	\text{ a.e.~on } {\partial \Omega}\}.
\end{multline*}

It is easy to check that 
$\varphi \in \A_B$
is a homeomorphism (by \cite[Theorem 2]{Ball1981}),
has finite distortion (as 
$J(x,\varphi) \geq 0$ 
a.e.)
and 
$\| K_I(\cdot,\varphi) \mid L_{s} (\Omega)\| \leq M$
by H\"older inequality
for
$s=\frac{\sigma r}{rn + \sigma - n} > 1$
where
$\sigma = \frac{q(1+m)}{q+m} > n$.
It means that 
$\A_B \cap \A(s,M;\overline\varphi) \neq \emptyset$.
Moreover,
a minimizing sequence
$\{\varphi_k\}\subset \A_B$
of the problem
\eqref{problem:example1}
belongs to 
$\A(s,M;\overline\varphi)$
as well.

On the other hand,
for the functions of the form~\eqref{example1}
Theorem~\ref{thm:main} holds.
Indeed,
$W_1(F)$
is polyconvex and satisfies 
\begin{equation*}
	W_1(F)\geq \alpha |F|^3 - \alpha,
\end{equation*}
where~%
$\alpha$
plays the role of the function
$h(x)$
of~\eqref{neq:coer}.
When we consider the~same boundary conditions 
$\overline{\varphi}\colon \overline{\Omega} \to \overline{\Omega'}$
and solve the minimization problem 
\begin{equation*}
	I_1(\varphi_0)=\inf \{I_1(\varphi) : \varphi\in\A(s,M;\overline\varphi)\}
\end{equation*}
Lemma~\ref{cor:ex_thm_dd} and Remark~\ref{rem:free_outer_dist} 
yields a~solution
$\varphi_0\in\A(s,M;\overline\varphi)$
which is a~homeomorphism.

Let us discuss another example.
Here the stored-energy function is of the~form
\begin{equation*}
	W_2(F)=a \, {\tr} (F^T F)^{\frac{3}{2}}.
\end{equation*}
This function is polyconvex and satisfies 
\begin{equation*}
	W_2(F)\geq \alpha \|F\|^3,
\end{equation*}
but violates the
inequality of the form~\eqref{neq:coer_b}.
Moreover,
$W_2(F)$
violates the~asymptotic condition
\begin{equation*}
	W_2(x,F)\to\infty
	\text{ as } \det F \to 0_+,
\end{equation*}
which plays an~important role in~\cite{Ball1981, BallCurOl1981}
and other articles.

Nevertheless,
for the stored-energy function
$W_2$
there exists a~solution to the minimization problem 
$	
I_2(\varphi_0)=\inf\limits I_2(\varphi)
$
in the class of homeomorhisms 
$\varphi\in\Hom(n-1,s,M)$,
$s>1$,
where
$I_2 (\varphi) = \int\limits_{\Omega} W_2 (D\varphi (x)) \, dx$.

\appendix

\section{Appendix, Geometry of domains}\label{sec:geometry}

It is known that the concept of a domain ``with Lipschitz boundary'' 
and a ``domain with quasi-isometric boundary'' are used in different senses.
To avoid ambiguity, we present in this section precise definitions
of such domains, used in the work, and their equivalence.

It is evident that 
the bi-Lipschitz mapping	 
is also a quasi-isometric one.
The inverse implication is not valid but 
the following assertion is true:
\textit{every quasi-isometric
mapping is locally bi-Lipschitz one} 
(see Lemma~\ref{lem:quasi-iso} below). 
Hence  
$\Omega$
is a domain 
with Lipschitz boundary (Definition~\ref{def:Lip_boundary})
if and only if it is a domain 
with quasi-isometric boundary (Definition~\ref{def:qi_boundary}).
Note that if the constant 
$M$
in Definition~\ref{def:quasiisom}
is allowed to depend on
$x$
and
$z$,
then a domain with quasi-isometric boundary may not have Lipschitz property nor cone property (see \cite[\S1.1.9]{Maz2011}).

\begin{definition}\label{def:quasiisom}
	A homeomorphism  
	$\varphi\colon U \to U'$ 
	of two open sets 
	$U$, 
	$U'\subset\R^n$ 
	is called a{\it~quasi-isometric mapping} 
	if the following inequalities 
	\begin{equation*}
		\varlimsup\limits_{y\to x}\frac{|\varphi(y)-\varphi(x)|}{|y-x|}\leq M \quad \text{ and } \quad 
		\varlimsup\limits_{y\to z}\frac{|\varphi^{-1}(y)-\varphi^{-1}(z)|}{|y-z|}\leq M
	\end{equation*}
	hold for all 
	$x\in U$ 
	and 
	$z\in U'$ 
	where  
	$M$ 
	is some constant independent of the~choice of points 
	$x\in U$ 
	and 
	$z\in U'$.
\end{definition}

\begin{definition}
	A mapping  
	$\varphi\colon U \to U'$ 
	of two open sets 
	$U$, 
	$U'\subset\R^n$ 
	is a{\it~bi-Lipschitz mapping} 
	if the following inequality 
	\begin{equation*}
		l |y-x| \leq |\varphi(y)-\varphi(x)| \leq L |y-x|
	\end{equation*}
	holds for all 
	$x$,
	$y\in U$  
	where  
	$l$
	and
	$L$ 
	are some constants independent of the~choice of points 
	$x$,
	$y\in U$.
\end{definition}

\begin{definition}\label{def:qi_boundary}
	A~domain
	$\Omega\subset\R^n$
	is called a~domain with {\it quasi-isometric  boundary}
	whenever for every point
	$x\in\partial \Omega$
	there are a~neighborhood
	$U_x \subset\R^n$
	and a~quasi-isometric mapping
	$\nu_x\colon U_x \to B(0, r_x) \subset \R^n$,
	where the~number
	$r_x>0$
	depends on~%
	$U_x$,
	such that
	$\nu_x (U_x \cap \partial \Omega) \subset \{y \in B(0, r_x) \mid y_n = 0\}$
	and
	$\nu_x (U_x \cap \Omega) \subset \{y \in B(0, r_x) \mid y_n > 0\}$.
\end{definition}

Let  
$d_{E}(u,v)$ 
denote the intrinsic metric in the domain 
$E$ 
defined as the infimum over the lengths of all rectifiable curves in 
$E$ 
with endpoints 
$u$ 
and~%
$v$.
It is well-known that a mapping is quasi-isometric if and only if the lengths of a rectifiable curve 
in the domain and of its image are comparable. 
The last property means the following one:  
\textit{given mapping 
$\varphi\colon \Omega \to \Omega'$ 
is quasi-isometric if and only if  
$ L^{-1}  d_B(x,y)\leq d_{\varphi(B)}(\varphi(x),\varphi(y))\leq L d_B(x,y)$ 
for all 
$x$,
$y\in B$.}

\begin{lemma}\label{lem:quasi-iso}
	Let 
	$\varphi\colon \Omega \to \Omega'$
	be a quasi-isometric
	mapping 
	then 
	for any fixed  ball 
	$B\Subset \Omega$ 
	the inequality
	$$
		d_{\varphi(B)}(\varphi(x),\varphi(y))\leq L|\varphi(x)-\varphi(y)|	
	$$ 
	holds for all points 
	$x$,
	$y\in B$ 
	with some constant 
	$L$ 
	depending on the choice of 
	$B$ 
	only.
\end{lemma}

\begin{proof}[Proof of Lemma~\ref{lem:quasi-iso}]
	Take an arbitrary function 
	$g\in W^1_\infty(\varphi(B))$. 
	Then 
	$\varphi^*(g) = g\circ \varphi\in W^1_\infty(B)$ 
	and, by the Whitney type extension theorem 
	(see for instance \cite{Vod1988,Vod1989}), 
	there is a bounded  extension operator 
	$\operatorname{ext}_B\colon W^1_\infty(B)\to W^1_\infty(\R^n)$. 
	Multiply 
	$\operatorname{ext}_B(\varphi^*(g))$ 
	by a cut-off-function 
	$\eta\in C_0^\infty(\Omega)$ 
	such that
	$\eta(x) =1$ 
	for all points 
	$x\in B$. 
	Then the product 
	$\eta\cdot\operatorname{ext}_B(\varphi^*(g))$ 
	belongs to 
	$W^1_\infty(\Omega)$, 
	equals~0 near the boundary 
	$\partial \Omega$ 
	and its norm in 
	$W^1_\infty(\Omega)$ 
	is controlled by the norm 
	$\|g\mid W^1_\infty(\varphi(B))\|$. 
	  
	It is clear that 
	$\varphi^{-1}{}^*(\eta\cdot\operatorname{ext}_B(\varphi^*(g)))$ 
	belongs to 
	$W^1_\infty(\Omega')$, 
	equals~0 near the boundary 
	$\partial \Omega'$ 
	and its norm in 
	$W^1_\infty(\Omega')$ 
	is controlled by the norm 
	$\|g\mid W^1_\infty(\varphi(B))\|$.  
	Extending 
	$\varphi^{-1}{}^*(\eta\cdot\operatorname{ext}_B(\varphi^*(g)))$ 
	by~0 outside~%
	$\Omega'$ 
	we obtain a~bounded extension operator
	$$
		\operatorname{ext}_{\varphi(B)}\colon W^1_\infty(\varphi(B))\to 
		W^1_\infty(\R^n).
	$$
	It is well-known (see for example \cite{Vod1988,Vod1989}) 
	that a necessary and sufficient condition for the existence of such 
	an operator is an~equivalence of the interior metric in  
	$\varphi(B)$ 
	to the Euclidean one:
	the inequality  
	$$
		d_{\varphi(B)}(u,v)\leq L |u-v|
	$$
	holds for all points 
	$u$,
	$v\in \varphi(B)$ 
	with some constant 
	$L$.
\end{proof}



\bibliographystyle{plain}
\bibliography{biblio_aei}

\end{document}